\documentclass[notitlepage,11pt,reqno]{amsart}
\usepackage{amssymb,amsmath,amscd,xspace,tocvsec2,mathrsfs}
\usepackage[breaklinks=true]{hyperref}
\usepackage[enableskew]{youngtab}
\usepackage[letterpaper]{geometry}
\DeclareMathOperator{\ch}{char}

\DeclareMathOperator{\Hom}{Hom}
\DeclareMathOperator{\Mod}{Mod}
\DeclareMathOperator{\End}{End}
\DeclareMathOperator{\Ext}{Ext}
\DeclareMathOperator{\Rad}{Rad}
\DeclareMathOperator{\Soc}{Soc}
\DeclareMathOperator{\im}{im}
\DeclareMathOperator{\id}{id}

\DeclareMathOperator{\wt}{wt}
\DeclareMathOperator{\coker}{coker}

\newcommand{\C}{\mathbb{C}}
\newcommand{\F}{\mathbb{F}}

\newcommand{\Z}{\mathbb{Z}}
\newcommand{\N}{\mathbb{N}}

\newcommand{\z}{\ensuremath{\widetilde{z}} }
\newcommand{\calo}{\mathcal{O}}

\newcommand{\calf}{\mathcal{F}}
\newcommand{\hfree}{\ensuremath{\widehat{H}^{free}}\xspace}
\newcommand{\ba}{\ensuremath{{\bf A}_{[\lambda]}} }
\newcommand{\bp}{\ensuremath{{\bf P}_{[\lambda]}} }

\newcommand{\cals}{\mathcal{W}}

\newcommand{\Y}{\ensuremath{\mathcal{YT}} }
\newcommand{\ycat}{\ensuremath{\mathscr{Y}_\blacktriangledown}}
\newcommand{\A}{\mathscr{A}}

\newcommand{\lie}[1]{\ensuremath{\mathfrak{#1}} }
\newcommand{\tangle}[1]{\ensuremath{\langle #1 \rangle}}
\newcommand{\mapdef}[1]{\ensuremath{\overset{#1}{\longrightarrow}}\xspace}
\newcommand{\comment}[1]{}

\newcommand{\kmz}{\ensuremath{\scriptstyle{k}} }
\newcommand{\kmo}{\ensuremath{\scriptscriptstyle{k-1}} }
\newcommand{\kmt}{\ensuremath{\scriptscriptstyle{k-2}} }
\newcommand{\jpo}{\ensuremath{\scriptscriptstyle{j+1}} }
\newcommand{\jpz}{\ensuremath{\scriptstyle{j}} }
\newcommand{\jmo}{\ensuremath{\scriptscriptstyle{j-1}} }
\newcommand{\svd}{\ensuremath{\scriptstyle{\vdots}} }
\newcommand{\sdd}{\ensuremath{\scriptstyle{\ddots}} }
\newcommand{\scd}{\ensuremath{\scriptstyle{\cdots}} }

\newif\iftextstyle
\textstyletrue
\everydisplay\expandafter{\the\everydisplay\textstylefalse}

\newtheorem{theorem}[equation]{Theorem}
\newtheorem{prop}[equation]{Proposition}
\newtheorem{lemma}[equation]{Lemma}
\newtheorem{cor}[equation]{Corollary}
\newtheorem{utheorem}{\textrm{\textbf{Theorem}} }

\theoremstyle{definition}
\newtheorem{defn}[equation]{Definition}
\newtheorem{stand}[equation]{Assumption}
\newtheorem{remark}[equation]{Remark}

\numberwithin{equation}{section}

\begin{document}
\title{On Category $\mathcal{O}$ over triangular Generalized Weyl
Algebras}
\author{Apoorva Khare}
\address[A.~Khare]{Department of Mathematics, Stanford University}
\email{\tt khare@stanford.edu}

\author{Akaki Tikaradze}
\address[A.~Tikaradze]{Department of Mathematics, University of Toledo}
\email{\tt tikar06@gmail.com}

\date{\today}
\subjclass[2010]{16S37, 18G05, 16D90}
\keywords{Triangular generalized Weyl algebra, Category $\calo$, Verma
module, projective module, tilting module, Koszul algebra, Ext groups,
STYT}

\begin{abstract}
We analyze the BGG Category $\mathcal{O}$ over a large class of
generalized Weyl algebras (henceforth termed GWAs).
Given such a ``triangular'' GWA for which Category $\mathcal{O}$
decomposes into a direct sum of subcategories, we study in detail the
homological properties of blocks with finitely many simples. As
consequences, we show that the endomorphism algebra of a projective
generator of such a block is quasi-hereditary, finite-dimensional, and
graded Koszul. We also classify all tilting modules in the block, as well
as all submodules of all projective and tilting modules. Finally, we
present a novel connection between blocks of triangular GWAs and Young
tableaux, which provides a combinatorial interpretation of morphisms and
extensions between objects of the block.
\end{abstract}
\maketitle

\settocdepth{section}
\tableofcontents

\section{Introduction and main results}\label{S1}

Generalized Weyl Algebras (GWAs) are an important and well-studied class
of algebras in the literature. There is much recent activity on the study
of GWAs, including their existence and consistency, structure, and
representation theory, as well as of special sub-families of GWAs. The
present paper provides a contribution to this area.

Recall \cite{Ba} that a GWA is generated over a ring $R$ (equipped with a
ring automorphism $\theta : R \to R$) by two elements $u,d$ with the
relations: $u r = \theta(r) u, r d = d \theta(r)$ for all $r \in R$, and
$ud = \theta(du) \in Z(R)$. We focus on the case when $R = H[du] = H[ud]$
for a commutative $\F$-algebra $H$ over a field $\F$; in the present
paper, (generalizations of) such algebras will be termed
\textit{triangular GWAs}.
These algebras enjoy several desirable properties, including a triangular
decomposition and an appropriate theory of weights. This allows the
introduction and study of the Bernstein-Gelfand-Gelfand (BGG) Category
$\calo$ over triangular GWAs.
Our goal in this paper is to show that a large amount of homological
information about Category $\calo$ can be obtained in a uniform manner
for all triangular GWAs. Specifically, given a weight $\lambda$ of $H$,
we study the endomorphism algebra $\ba = \End_\calo(\bp)^{op}$ of a
specific projective generator $\bp$ of the corresponding block
$\calo[\lambda]$ of $\calo$. As a first step, a general treatment of
Category $\calo$ can be used to show that when the block has finitely
many simple objects, the algebra $\ba$ is $\Z_+$-graded, associative,
finite-dimensional, and quasi-hereditary.

In this paper we compute all Ext-groups for pairs of simple modules,
Verma modules, or (quotients of) projective modules, as well as the
Ext-groups between these modules, for a general triangular GWA. Our
results yield many desirable homological consequences for blocks of
triangular GWAs. First, we provide a presentation for the algebras $\ba$
and show that the isomorphism class of the algebra depends only on the
(finite) number of simple objects in the block. In particular, this shows
that all (finite) blocks of triangular GWAs with equal numbers of simple
objects are Morita equivalent.

Second, we prove that the algebras $\ba$ are Koszul. Koszulity is an
important structural property for $\Z_+$-graded, quadratic algebras and
has several desirable homological consequences; see
e.g.~\cite{BGS,CS2,PP,Pr} for more on Koszulity and its generalizations.

An additional consequence is a complete description of all tilting
modules in blocks of Category $\calo$, as well as an enumeration of all
submodules of projective or tilting objects in a block. Specifically, we
show that each such submodule is indecomposable and has a Verma flag.

A fourth consequence is an interesting and novel connection to Young-type
tableaux, which to our knowledge has not been explored in the literature.
These tableaux satisfy combinatorial counterparts of our homological
results, as we explain in this paper. In other words, blocks of $\calo$
categorify Young tableaux.

Finally, the complete and explicit descriptions afforded by our
computations make it possible to apply the comprehensive homological
machinery developed by Cline, Parshall, and Scott in their broad program
for highest weight categories. For instance, we show as a corollary of
our results that the blocks of Category $\calo$ satisfy the
\textit{Strong Kazhdan-Lusztig condition (SKL)} as in \cite{CPS2}.

\subsection{Triangular GWAs}

We now develop the notation required to present the main results later in
this section. We begin by introducing the main object of study in the
present paper. For this paper we fix an arbitrary ground field $\F$;
thus, $\dim$ henceforth denotes $\dim_\F$. Also, let $\Z_+$ denote the
set of non-negative integers.

\begin{defn}
Suppose $H$ is an commutative $\F$-algebra with an $\F$-algebra
automorphism $\theta : H \to H$, and elements $z_0 \in H, z_1 \in
H^\times$. The \textit{triangular Generalized Weyl Algebra (triangular
GWA)} associated to this data is defined to be the $\F$-algebra
\begin{equation}\label{Egwareln}
\cals(H,\theta,z_0,z_1) := H \tangle{d,u} / (uh = \theta(h) u,\ hd = d
\theta(h),\ ud = z_0 + d z_1 u\ \forall h \in H).
\end{equation}
\end{defn}

Triangular GWAs are the focus of a concerted research effort in the
literature. A large class of triangular GWAs that has been the focus of
much recent research consists of \textit{down-up algebras}.
These are a family of generalized Weyl algebras that occur in several
different settings, including representation theory, mathematical
physics, Hopf algebras, ring theory, and combinatorics. See
\cite{JZ,Ku1,LeB,Ro,Smi,Ta1,Wi} for these and other motivations.
It turns out that the algebras in the above references have certain
common structure and properties. For instance, they contain elements $d$
and $u$ that should be thought of as ``down" (lowering) and ``up"
(raising) operators. In order to systematically study their behavior,
Benkart and Roby \cite{BR} defined \textit{down-up algebras} and
initiated their study.
Since then, down-up algebras and their variants have been the focus of
tremendous interest - to name a few references, see
\cite{CM,CS1,Jo2,KM,Ku2,LL,Ru,Zha}. Other examples of down-up algebras
have been studied by Woronowicz \cite{Wo}, as well as Kac in the
comprehensive work \cite{Kac1} on Lie superalgebras. We remark that
down-up algebras are a sub-family of triangular GWAs with $H = \F[h]$, a
polynomial algebra; see \cite[Section 8]{Kh2} for more details.

Simultaneously, another area of much recent interest is the study of
various ``quantum" and Hopf-like algebras. These ``quantum" variants are
generated by $u,d$ over the group ring $\F[\Gamma]$ for some group
$\Gamma$. As above, examples have arisen from a variety of settings,
including Kleinian singularities and quantum groups. See
\cite{CBH,JWY,JWZ,Ta2,TX,Zhi} for more references. As above, all of these
``quantum" down-up algebras are triangular GWAs with $H = \F[K^{\pm 1}]$,
a group algebra - see the discussion in \cite[Section 8]{Kh2}. It is also
shown in \textit{loc.~cit.}~that the down-up algebras in the former,
``classical" family, admit quantizations that belong to the latter,
``quantum" family.

Both classical and quantum down-up algebras are special cases of
\textit{ambiskew polynomial rings}, which are the class of triangular
GWAs where $z_1 \in \F^\times$. Ambiskew polynomial rings are the focus
of recent and continuing interest \cite{BM,Ha,Jo1,JW}. Generalized Weyl
algebras can also arise from other constructions. For instance as
explained in \cite[Section 9]{Kh2}, continuous Hecke algebras of $GL(1)$
and $\C \oplus \C^*$ (see \cite{EGG}) are generalized Weyl algebras.
Thus our goal in the present paper is to prove results for general
triangular GWAs, addressing uniformly all of the above examples.

\subsection{Category $\mathcal{O}$}

In order to state the main results in this paper, we now introduce a
sequence $\z_n$ of distinguished elements in a triangular GWA (more
precisely, in its ``Cartan subalgebra" $H$). We also set further
notation.

\begin{defn}\label{Delements}
Fix an $\F$-algebra $H$, elements $z_0 \in H$ and $z_1 \in H^\times$, and
an $\F$-algebra automorphism $\theta : H \to H$.
Now let $A := \cals(H,\theta,z_0,z_1)$ be the algebra defined as in
Equation \eqref{Egwareln}.
\begin{enumerate}
\item Given an integer $n \geq 1$, define
\begin{equation}
z'_n := \prod_{i=0}^{n-1} \theta^i(z_1), \qquad z'_0 := 1, \qquad
\z_n := \sum_{i=0}^{n-1} \theta^i(z_0 z'_{n-1-i}), \qquad \z_0 := 0, \qquad
\z_{-n} := \theta^{-n}(\z_n).
\end{equation}

\item Define a \textit{character} or \textit{weight} of $H$ to be an
$\F$-algebra map $: H \to \F$, and denote the set of weights of $H$ by
$\widehat{H} := \Hom_{\F-alg}(H,\F)$. Now given a weight $\lambda \in
\widehat{H}$, define
\begin{align}
[\lambda] & := \{ \lambda \circ \theta^n : n \in \Z, \lambda(\z_n) = 0
\},\\
\hfree & := \{ \lambda \in \widehat{H} :\ \forall n \in \Z
\setminus \{ 0 \},\ \exists h \in H \mbox{ with } \lambda(h) \neq
(\lambda \circ \theta^n)(h) \}.\notag
\end{align}

\item Given an $H$-module $M$ and $\lambda \in \widehat{H}$, the
\textit{$\lambda$-weight space} of $M$ is $M_\lambda := \{ m \in M :
\ker(\lambda) m = 0 \}$. Now define $\wt M := \{ \lambda \in \widehat{H}
: M_\lambda \neq 0 \}$. We say that $M$ is a \textit{weight module} over
$H$ if $M = \bigoplus_{\lambda \in \widehat{H}} M_\lambda$.

\item Define the \textit{BGG Category $\calo$} over $A$ to be the full
subcategory of all finitely generated $H$-weight $A$-modules, with
finite-dimensional $H$-weight spaces and a locally finite action of $u$.

\item We show in Remark \ref{Rverma} below that Category $\calo$ contains
pairwise non-isomorphic simple objects $L(\lambda)$ for all $\lambda \in
\hfree$. Now given a subset $T \subset \hfree$, define $\calo(T)$ to be
the full subcategory of all objects in $\calo$, each of whose
Jordan-Holder subquotients is $L(\lambda)$ for some $\lambda \in T$. Also
let $\calo_\N$ denote the full subcategory of all finite length objects
in $\calo$, and define $\calo_\N(T) := \calo_\N \cap \calo(T)$. If $T =
[\lambda]$, denote $\calo[\lambda] := \calo([\lambda])$ and
$\calo_\N[\lambda] := \calo_\N([\lambda])$.
\end{enumerate}
\end{defn}

It is then clear that $\Z$ acts on the set of weights $\lambda : H \to
\F$ (and this action is free on the subset $\hfree$), via: $n * \lambda
:= \lambda \circ \theta^{-n}$. This yields a partial order on
$\widehat{H}$, via:
$\lambda < n * \lambda$ for all $n > 0$ and $\lambda \in \hfree$.
Throughout this paper we will use the following (slight) abuse of
notation without further reference: $\lambda - \mu = n \in \Z$ if $n *
\mu = \lambda$ for $\lambda, \mu \in \hfree$. The following identity is
also useful in this setting, and easily verified:
\begin{equation}\label{ES3}
\z_{m+n} = \z_n \theta^n(z'_m) + \theta^n(\z_m),
\qquad \forall n,m \geq 0.
\end{equation}

\noindent From \eqref{ES3}, it follows easily that $[\lambda] = [\mu]$
whenever $\mu \in [\lambda]$.

\begin{remark}
We briefly elaborate on the set $[\lambda]$. It turns out that $\mu \in
[\lambda]$ if and only if either $[M(\lambda) : L(\mu)] > 0$ or $[M(\mu)
: L(\lambda)] = 0$. In other words, if $\lambda > \mu$, then $u \cdot
d^{\lambda - \mu} m_\lambda = 0$. This is akin to a maximal/primitive
vector for the positive nilpotent Lie subalgebra $\lie{n}^+$, for a
semisimple Lie algebra. Another way to view $[\lambda]$, if $z_1 = 1$ and
$z_0 \in \im(\id_H - \theta)$, is as follows: under these assumptions $A$
has a central Casimir operator $\Omega$ (see \cite[Section 8]{Kh2}); then
$[\lambda]$ is precisely the set of weights $\mu \in \Z * \lambda$ for
which the central characters on the Verma modules $M(\mu), M(\lambda)$
coincide.
\end{remark}

\subsection{Main results}

To state our main results, we require the notion of a \textit{Koszul
algebra}, which is a useful homological property in the study of
quadratic algebras \cite{PP}, prominent in representation theory.

\begin{defn}[{\cite[Definition 1.1.2]{BGS}}]
A ring $A$ is said to be \textit{Koszul} if it satisfies the following
conditions:
\begin{enumerate}
\item $A$ is a $\Z_+$-graded ring, with $A_0$ a semisimple subring.

\item The graded left $A$-module $A_0$ admits a graded projective
resolution
\[
\cdots \to P^2 \to P^1 \to P^0 \twoheadrightarrow A_0
\]

\noindent such that $P^i$ is generated by its degree $i$ component, i.e.,
$P^i = A P^i_i\ \forall i \geq 0$.
\end{enumerate}

\noindent Also define, for a $\Z_+$-graded ring $A = \bigoplus_{j \geq 0}
A_j$, its \textit{homological dual} $E(A) := \Ext^\bullet_A(A_0, A_0)$.
\end{defn}

Similarly, one also defines for a quadratic $\F$-algebra $A = T(V) / (Q)$
(with $Q \subset V \otimes V$), its \textit{quadratic dual} $A^! :=
T(V^*) / (Q^\perp)$. Then the following properties of Koszul algebras are
well-known.

\begin{theorem}[{\cite[Section 2]{BGS}}]\label{Tbgs}
Suppose $A$ is a finite-dimensional Koszul algebra over a field $\F$.
Then $A$ is quadratic. Moreover, $E(A)$ is also Koszul, and is isomorphic
as an $\Z_+$-graded $\F$-algebra to $(A^!)^{op}$, whence $E(E(A)) \cong
A$.
\end{theorem}

Using the above notation, it is possible to state the first main result
of this paper.

\comment{
$\spadesuit$ Can we prove this result for \textit{any} uniserial block of
Category $\calo$ for \textit{any} regular triangular algebra? At least,
say, one where $N_+$ is commutative? Or more generally, in which there
exists a root vector $u \in N_+$ such that the Verma subquotients of
projective modules $P(\lambda,l)$ all have highest weight vectors given
by powers of $u$.
}

\begin{utheorem}\label{Tkoszul}
Suppose $\cals(H,\theta,z_0,z_1)$ is a triangular GWA, for which
$\hfree$ is non-empty and Category $\calo$ is finite length.
Suppose $[\lambda]$ is finite for some $\lambda \in \hfree$.
\begin{enumerate}
\item Let $\{ L_i : 1 \leq i \leq n = |[\lambda]| \}$ be the set of
simple objects in $\calo[\lambda]$, and $P_i$ be the projective cover of
$L_i$ in the block. Then $\calo[\lambda]$ is equivalent to $\ba-\Mod$,
where $\ba = \End_\calo(\bigoplus_i P_i)^{op}$ is a $\Z_+$-graded
$\F$-algebra of dimension $1^2 + \cdots + n^2$, which is quasi-hereditary
and Koszul.

\item The $\Ext$-quiver of $\ba$ is the double $\overline{A_n}$ of the
$A_n$-quiver
\[
[1] \to [2] \to \cdots \to [n].
\]

\noindent Label the arrows as $\gamma_i : [i+1] \to [i]$ and $\delta_i :
[i] \to [i+1]$. Then $\ba^{op}$ is isomorphic to the path algebra of the
quiver $\overline{A_n}$ with relations
\begin{equation}\label{Erelations}
\delta_i \circ \gamma_i = \gamma_{i+1} \circ \delta_{i+1} \ \forall
0<i<n-1, \qquad \delta_{n-1} \circ \gamma_{n-1} = 0.
\end{equation}
\end{enumerate}
\end{utheorem}

\noindent Thus at its heart, Category $\calo$ over every triangular GWA
(with commutative $H$) is governed by a distinguished family of
finite-dimensional Koszul algebras $\ba$, which may be denoted by $\A_n$
to denote their dependence only on the integer $n = |[\lambda]| \geq 1$.
In particular, all finite blocks of Category $\calo$ over any triangular
GWA, having exactly $n$ simple objects, are Morita equivalent to
finite-dimensional $\A_n$-modules. We also remark that the algebras
$\A_n$ have connections to other settings in representation theory; see
Remark \ref{Rkoszul} for more details.\medskip

Note that Theorem \ref{Tkoszul} holds for a very large class of
generalized Weyl algebras. For instance, it has the following consequence
that applies to a large class of algebras described above in this
section.

\begin{cor}\label{Ckoszul}
Suppose $\ch \F = 0$, and $A = \cals(\F[h],\theta,s^{-1} f(h), s^{-1})$ is
a generalized down-up algebra with $r = 1$, $\gamma \in \F^\times$, and
$0 \not\equiv f \in \F[h]$. Also suppose that $s=1$ or $s$ is not a root
of unity. Then:
\begin{enumerate}
\item Category $\calo$ over $A$ has a block decomposition into summands
$\calo[\lambda]$.
\item Every block $\calo[\lambda]$ contains only finitely many
non-isomorphic simple objects.
\item For each block $\calo[\lambda]$, the corresponding algebra $\ba$ is
Koszul.
\end{enumerate}

\noindent Similar results also hold for ``quantum" analogues of such
algebras (mentioned above). Namely, suppose
\[
\ch \F = 0, \qquad \Gamma = \tangle{K,K^{-1}} \cong \Z, \qquad \theta(K)
= qK, \qquad z_1 \in \F^\times,
\]

\noindent with $q \in \F^\times$ not a root of unity. Also suppose that
$z_0 \in \F[K^{\pm 1}]$ is not of the form $b K^n$ for any $b \in \F, n
\in \Z$. Then the three assertions above (in this corollary) hold for $A
= \cals(H,\theta,z_0,z_1)$.
\end{cor}

In particular, Theorem \ref{Tkoszul} holds for Smith's family of algebras
\cite{Smi} with $[x,y] = f(h) \neq 0 = \ch \F$, as well as for the
``quantized version" of Smith's algebras studied by Ji et al.~\cite{JWZ}
and Tang \cite{Ta2}, as long as $q$ is not a root of unity and $z_0
\notin \bigcup_{n \in \Z} \F K^n$.\medskip

We also remark that Theorem \ref{Tkoszul} can be proved for an even
larger class of algebras with triangular decomposition. See Remark
\ref{Rtgwa}.

The heart of the proof of Theorem \ref{Tkoszul} involves homological
calculations in Category $\calo$ over a triangular GWA. This leads to our
next main result.

\begin{utheorem}\label{Thom}
(Setting as in Theorem \ref{Tkoszul}.)
Suppose $[\lambda] = \{ \lambda_1 < \lambda_2 < \cdots < \lambda_n \}
\subset \hfree$. Then for all $1 \leq i,j \leq n$ and $l>0$,
\begin{equation}
\Ext^l_\calo(L(\lambda_i), L(\lambda_j)) =
  \begin{cases}
    \F,	&\text{if $|i-j| = l = 0$;}\\
    \F,	&\text{if $|i-j| = l = 1$;}\\
    \F,	&\text{if $i=j \neq 1$ and $l=2$;}\\
    0,	&\text{otherwise.}
  \end{cases}
\end{equation}

\noindent Now define for $1 \leq i \leq n$:
\begin{equation}\label{Enotation}
L_i := L(\lambda_i), \qquad M_i := M(\lambda_i), \qquad P_i :=
P(\lambda_i), \qquad P_{n+1} := 0 =: L_0,
\end{equation}

\noindent where $P(\lambda)$ denotes the projective cover of $L(\lambda)$
in $\calo[\lambda]$, and $M(\lambda) = A / (A \cdot u + A \cdot \ker
\lambda)$ is the ``Verma module" of highest weight $\lambda$.
Then for all $1 \leq i \leq n$, $M_i$ has a finite filtration
\[
M_i \supset M_{i-1} \supset \cdots \supset M_1 \supset 0 =: M_0,
\]

\noindent with successive subquotients $L_j$ for $1 \leq j \leq i$.
Similarly, every $P_i$ has a ``Verma flag"
\[
P_i \supset P_{i+1} \supset \cdots \supset P_n \supset 0,
\]

\noindent with successive subquotients $M_j$ for $i \leq j \leq n$.
Moreover, for all $1 \leq j < k \leq n+1$ and $0 \leq s < r \leq n$,
defining ${\bf 1}(E)$ for a mathematical condition $E$ to be $1$ when the
condition $E$ holds, and $0$ otherwise, we have:
\begin{align}\label{Eext}
\dim \Ext^l_\calo(M_r, P_j / P_k) = &\ \delta_{l,0} {\bf 1}(r < k) +
\delta_{l,1} {\bf 1}(r < j),\\
\dim \Ext^l_\calo(P_j / P_k, M_r / M_s) = &\ \delta_{l,0} {\bf 1}(s < j
\leq r) + \delta_{l,1} {\bf 1}(s < k \leq r).\notag
\end{align}
\end{utheorem}

Theorem \ref{Thom} summarizes important homological information in the
block $\calo[\lambda]$. For instance, in the special case $k=j+1$,
Theorem \ref{Thom} computes all $\Ext$-groups between Verma modules and
highest weight modules.

Recall that the definition of Koszulity involves the Ext-algebra
$E(\ba)$. Our next main result involves understanding the structure of
$E(\widetilde{\ba})$, where $\widetilde{\ba}$ is the larger algebra
given by
\begin{equation}
\widetilde{\ba}^{op} := \End_\calo \bigoplus_{1 \leq r < s \leq n+1} P_r
/ P_s.
\end{equation}

\noindent In turn, this enables a detailed analysis of projective objects
in the highest weight category $\calo[\lambda]$, as well as a complete
classification of indecomposable injective and tilting modules (i.e.,
modules that have both a Verma flag as well as a dual Verma flag).

\begin{utheorem}\label{Tproj}
Setting as in Theorems \ref{Tkoszul} and \ref{Thom}.
\begin{enumerate}
\item Fix integers $1 \leq j < k \leq n+1$ and $1 \leq r < s \leq
n+1$. Then,
\begin{equation}\label{Eproj-dim}
\dim \Ext^l_\calo(P_r / P_s, P_j / P_k) =
  \begin{cases}
    \mathbf{1}(r<k) \min(s-r,k-r,k-j),	&\text{if $l = 0$;}\\
    \mathbf{1}(r \leq j) \mathbf{1}(s \leq k) (\min(0,j-s) +
    \min(s-r,k-j)), &\text{if $l = 1$;}\\
    0,	&\text{otherwise.}
  \end{cases}
\end{equation}

\item Given $1 \leq r < s \leq n+1$, there exists a bijection between the
submodules of $P_r / P_s$, and strictly decreasing sequences of integers
$s-1 \geq m_l > m_{l-1} > \cdots > m_1 \geq 1$, for some $0 \leq l \leq
s-r$.
Every such submodule is indecomposable and has a Verma flag, and the
number of these submodules is $\displaystyle \sum_{l=0}^{s-r}
\binom{s-1}{l}$.

\item The partial/indecomposable tilting modules in the block
$\calo[\lambda]$ are precisely $T_k := P_1 / P_{k+1}$ for $1 \leq k \leq
n$. Each of these modules is self-dual.
In particular, the injective hull in the block $\calo[\lambda]$ of the
simple module $L_k$ is equal to $F(P_k) \cong T_n / T_{k-1}$, where we
set $T_0 := 0$.
\end{enumerate}
\end{utheorem}

\begin{remark}
The condition that $[\lambda] \subset \hfree$ is a natural one to assume.
In the special case of $\cals(H, \theta, z_0, z_1) = U(\lie{sl}_2)$, the
condition amounts to requiring that $\F$ has characteristic zero, while
for $\cals(H, \theta, z_0, z_1) = U_q(\lie{sl}_2)$, the condition amounts
to $q$ not being a root of unity. Thus, this condition affords a
``clean'' picture in the case of a general triangular GWA, and allows us
to focus on the technical issues of Koszulity and the structure of $\ba,
\widetilde{\ba}$.
\end{remark}

Observe that our main results do not make any assumption on the ground
field (other than $\hfree$ being non-empty, which can entail $\ch \F =
0$).
In particular, we do not require $\F$ to be algebraically closed, as is
the case in the literature when methods involving Gabriel's theorem are
used, to discuss the structure of basic, finite-dimensional, Koszul
algebras. In this paper we do not use Gabriel's result, but rather, rely
on the comprehensive homological information that we derive about the
algebras $\ba$ and $\widetilde{\ba}$ from Theorems \ref{Thom} and
\ref{Tproj}. Thus, we will first prove Theorems \ref{Thom} and
\ref{Tproj}, and then use these results to show the Koszulity and
structure of the algebra $\ba$ in Theorem \ref{Tkoszul}.

Finally, a novel feature of this paper involves introducing an
appropriate combinatorial category of Young tableaux. This is carried out
in Section \ref{Sstyt}, where we provide strong and novel homological
connections between this category and all finite blocks $\calo[\lambda]$
for an arbitrary triangular GWA.

\subsection*{Organization of the paper}
The remainder of this paper is organized as follows. In Section
\ref{Sprelim} we recall the standard approach for developing a theory of
Category $\calo$ over a triangular GWA, leading up to the block
decomposition of $\calo$ into highest weight categories. In Section
\ref{Sproj} we prove a projective resolution of any simple module in
$\calo[\lambda]$ and also prove Theorem \ref{Thom}. Next, in Section
\ref{Stilting} we study maps between the modules $P_r / P_s$, i.e.~the
algebra $\widetilde{\ba}$. Using this we classify all tilting modules,
projective modules, and their submodules. This helps in proving Theorem
\ref{Tproj}, and in Section \ref{S5}, Theorem \ref{Tkoszul}. Finally, in
Section \ref{Sstyt} we define and study sub-triangular Young tableaux
(STYTs), and their many homological connections to the block
$\calo[\lambda]$.

\section{PBW decomposition and the Bernstein-Gelfand-Gelfand
Category}\label{Sprelim}

In this section, we list certain basic properties of triangular GWAs as
well as Category $\calo$ over them. These properties will be used in
proving our main results in the subsequent sections of the paper.

\subsection{PBW property}

We begin with a few preliminary observations on triangular GWAs. The
results in this subsection are not hard to show, and we omit their
proofs as they are relatively straightforward computations. The first
observation is that if $z_1$ is invertible in $H$, then the triangular
GWA $\cals(H,\theta,z_0,z_1)$ is in fact a generalized Weyl algebra over
$H[ud]$, with $\theta$ extended to $H[ud]$ via: $\theta(ud) = ud z_1 +
\theta(z_0)$. This is made more precise in the following result.

\begin{lemma}\label{Lgwa}
Suppose $H$ is an $\F$-algebra with automorphism $\theta$ and $z_0, z_1
\in H$. Define $A = \cals(H,\theta,z_0,z_1)$ as in Equation
\eqref{Egwareln}.
\begin{enumerate}
\item Then $ud, du$ commute with all of $H$. Moreover, $H \tangle{ud} = H
\tangle{du}$.

\item $du$ and $ud$ are simultaneously algebraic or simultaneously
transcendental over $H$ (in $A$).

\item If $du$ is transcendental over $H$ (in $A$), then the following are
equivalent:
\begin{enumerate}
\item $\theta$ extends to an $\F$-algebra automorphism of $H \tangle{du}
= H[du] = H[ud]$, and $A$ is a Generalized Weyl Algebra of degree $1$
over $H[du]$, with $\theta(du) = ud = z_0 + d z_1 u$.

\item $z_1 \in H^\times$ is a unit in $H$.
\end{enumerate}
\end{enumerate}
\end{lemma}

We next discuss a useful characterization of the transcendence of the
elements $du,ud$ over $H$ in a triangular GWA. This characterization,
called the \textit{PBW property}, allows us to work with a distinguished
$\F$-basis, and is explained as follows. A triangular GWA
$\cals(H,\theta,z_0,z_1)$ is equipped with a $\Z_+$-filtration that
assigns degree $0$ to $H$ and degree $1$ to $d,u$. The associated graded
algebra is the (possibly non-commutative) algebra
$\cals(H,\theta,0,z_1)$. A natural question is to classify all of the
flat -- or \textit{PBW} -- deformations $\cals(H,\theta,z_0,z_1)$. Recall
that flat deformations can be characterized in terms of Ore extensions
$S[X;\sigma,\delta]$, where $\sigma$ is an algebra automorphism of the
$\F$-algebra $S$, and $\delta$ is a $\sigma$-derivation of $S$. Now note
that $H$ and $u$ generate a semidirect product algebra $H \ltimes \F[u]$.
Then the following result is not hard to show, and is used without
reference throughout the remainder of the paper.

\begin{theorem}[PBW property]\label{Tpbw}
Suppose $H$ is an $\F$-algebra with automorphism $\theta$ and $z_0, z_1
\in H$. Define $A = \cals(H,\theta,z_0,z_1)$ as in Equation
\eqref{Egwareln}. Then the following are equivalent:
\begin{enumerate}
\item $\cals(H,\theta,z_0,z_1)$ is a flat deformation of
$\cals(H,\theta,0,z_1)$. (This is called the ``PBW property''.)

\item $z_0, z_1$ are central in $H$.

\item The maps $\sigma,\delta : H \ltimes \F[u] \to H \ltimes \F[u]$
given by
\[
\sigma(u) = z_1 u, \qquad \sigma|_H \equiv \theta, \qquad \delta(u) =
z_0, \qquad \delta|_H \equiv 0
\]

\noindent are indeed an algebra automorphism and a $\sigma$-derivation
respectively.
\end{enumerate}

\noindent In particular, $\cals(H,\theta,z_0,z_1)$ is an Ore extension if
these (equivalent) conditions hold:
\begin{equation}\label{Eskewder}
\cals(H,\theta,z_0,z_1) = H[u; \theta^{-1}, 0][d;\sigma,\delta].
\end{equation}

\noindent If, moreover, $z_1$ is not a zero-divisor in $H$, then these
conditions are equivalent to:
\begin{enumerate}
\setcounter{enumi}{3}
\item $ud, du$ are transcendental over $H$.
\end{enumerate}
\end{theorem}

\noindent Note that such a deformation would have a ``PBW'' $\F$-basis
$\{ d^r h_i u^s \ : \ 0 \leq r,s \in \Z, \ i \in I \}$,
where $\{ h_i : i \in I \}$ runs over an $\F$-basis of $H$. In
\cite[Section 8]{Kh2}, it was explored if the aforementioned examples of
triangular GWAs satisfied the assumptions of Theorems \ref{Tpbw} and
\ref{Tkoszul}.

In the proof of $(1) \implies (4)$ in Theorem \ref{Tpbw}, certain
computations are used that are also needed later in this paper. We now
state these computations for future use.

\begin{prop}\label{Psmith1}
Suppose $H$ is an $\F$-algebra, with automorphism $\theta$ and $z_0, z_1
\in Z(H)$ central. Define $A = \cals(H,\theta,z_0,z_1)$ as in Equation
\eqref{Egwareln}.
\begin{enumerate}
\item The centralizers in $H$ of $u$ and $d$ coincide: $Z_H(u) = Z_H(d) =
\ker (\id_H - \theta)$.

\item For all $h, h_1, \dots, h_n \in H$ and integers $0 \leq m \leq n$,
\begin{align}
u \cdot d^n h \quad = & \quad d^n \theta(h) z'_n u + d^{n-1} \z_n
h,\notag\\
\prod_{i=1}^n (d h_i u) \quad \in & \quad d^n \theta^{n-1}(h_1 \cdots h_n)
\prod_{i=0}^{n-1} \theta^i(z'_{n-1-i}) u^n + \sum_{i=1}^{n-1} d^i \cdot H
\cdot u^i, \label{Ecomp}\\
u^m d^n \quad \in & \quad d^{n-m} \cdot \prod_{j=n-m}^{n-1} \z_{j+1} +
\cals(H,\theta,z_0,z_1) \cdot u.\notag
\end{align}

\item For all $i,j,k,l \in \Z_+$ and $h,h' \in H$, $d^i h u^j \cdot d^k
h' u^l \in \sum_{t=0}^{\min(j,k)} d^{i+k-t} \cdot H \cdot u^{j+l-t}$.
\end{enumerate}
\end{prop}

\noindent The proofs of these statements are standard and are hence
omitted.

In proving Theorem \ref{Thom}, we require one further preliminary result
in $\cals(H,\theta,z_0,z_1)$.

\begin{prop}\label{Pakaki}
Suppose $A = \cals(H,\theta,z_0,z_1)$ is a triangular GWA.
Consider the grading on $A$ with $\deg u = 1, \deg d = -1, \deg H = 0$.
Then for all $n \in \Z$, $A[n]$ is isomorphic to $H[du]$ as a left
$H[du]$-module.
\end{prop}

\begin{proof}
It follows from Theorem \ref{Tpbw} that the $n$th graded component $A[n]$
of $A$ is spanned by $d^m H u^{m+n}$ for all $m \geq \max(0,-n)$. It
follows from the PBW Theorem \ref{Tpbw} that the result for $n \geq 0$
reduces to that for $n = 0$. It thus suffices to show the result for $n
\leq 0$.

First suppose that $n=0$, and consider the identity map between the
filtered vector spaces $: H[du] \to A[0]$, where the filtration is
according to the length of the monomials in $d,u$. By the second of the
equations \eqref{Ecomp} (with $h_i = 1$ for all $i$), this map is an
isomorphism on each filtered piece (given by an invertible triangular
matrix, since $z_1 \in H^\times$). This shows the result for $n=0$.
Next, if $n<0$, note that $A[n] = d^{-n} A[0] \cong d^{-n} H[du]$ is a
free rank one right $H[du]$-module. Thus, it remains to show that $d^{-n}
A[0] \cong A[0] d^{-n}$ for $n < 0$. This can be shown using the first of
the equations \eqref{Ecomp} (with $h=1$), the filtration on $A[0]$, and
that $z_1 \in H^\times$.
\end{proof}

\subsection{The Bernstein-Gelfand-Gelfand Category}

The goal of this subsection is to introduce and develop basic properties
of an important category of weight modules of triangular GWAs -- an
analogue of the Bernstein-Gelfand-Gelfand category $\calo$ \cite{BGG1}.
In light of Lemma \ref{Lgwa} and Theorem \ref{Tpbw}, we make the
following assumptions.

\begin{stand}\label{Stand2}
For the remainder of this paper, assume that $A =
\cals(H,\theta,z_0,z_1)$ is a triangular GWA for which $\hfree$ is
non-empty.
\end{stand}

These assumptions are satisfied by many of the examples in the literature
when $\theta$ is an algebra automorphism of $H$. See \cite[Section
8]{Kh2} for the two large ``classical'' and ``quantum'' families of
examples. Also note here that $\theta$ is necessarily not of finite
order.

We now define and study Category $\calo$ via a series of results that are
required in future sections. We omit the proofs as these results are
shown in \cite{Kh2} in greater generality. We begin by setting some
notation.

\begin{defn}
Define the \textit{Verma module} with highest weight $\lambda \in
\widehat{H}$ to be $M(\lambda) := A / (Au + A \cdot \ker \lambda)$.
\end{defn}

\begin{remark}\label{Rverma}
We now list standard properties of Verma modules and $\calo$; see
\cite{Kh2} for the proofs.
\begin{enumerate}
\item Given $n \geq 0$, $\lambda \in \widehat{H}$, and an $A$-module $M$,
we have: $u^n M_\lambda \subset M_{n * \lambda}$ and $d^n M_\lambda
\subset M_{(-n) * \lambda}$.

\item For all $\lambda \in \hfree$, $M(\lambda) \in \calo$. It is a
weight module with all nonzero weight spaces of weight $n * \lambda$ for
some $n \leq 0$.

\item $M(\lambda)$ is generated by a weight vector $m_\lambda$, as a free
$\F[d]$-module. The weight space $M_{(-n) * \lambda}$ is spanned by $d^n
m_\lambda$, and $u m_\lambda = 0$. From the point of view of $A$ being a
generalized Weyl algebra (Lemma \ref{Lgwa}), $M(\lambda)$ is killed by
$\ker_{H[du]}(\lambda)$, where every $\lambda \in \widehat{H}$ extends to
an algebra map $\lambda : H[du] \to \F$ that sends $du$ to $0$.

\item For $\lambda \in \hfree$, $M(\lambda)$ has a unique simple quotient
$L(\lambda)$. The modules $L(\lambda)$ are pairwise non-isomorphic.

\item $A$ has an anti-involution $i : A \to A$, which sends $d
\leftrightarrow u$ and fixes all of $H$.

\item The anti-involution $i$ induces a contravariant, involutive,
additive ``restricted duality functor'' $F$, which preserves the category
$\calo_\N$ of finite length objects in $\calo$. It is defined as follows:
given an $H$-weight module $M := \oplus_{\mu \in \widehat{H}} M_\mu$, we
define $F(M) := \oplus_{\mu \in \widehat{H}} M_\mu^*$. This is an
$A$-module action via: $(am^*)(m) := m^*(i(a)m)$ for $a \in A, m \in M,
m^* \in F(M)$. Then,
\[
F(L(\lambda)) = L(\lambda)\ \forall \lambda \in \hfree, \qquad F^2(M)
\cong M\ \forall M \in \calo_\N.
\]

\item Moreover, $F$ is an exact functor on $\calo_\N$.
\end{enumerate}
\end{remark}

We now discuss the structure of Category $\calo$, which turns out to be
somewhat different from the well-studied case of Lie algebras with
triangular decomposition \cite{MP}.

\begin{prop}\label{Pverma2}
Every module $M \in \calo$ is a direct sum of summands:
\[
M = \bigoplus_{\tangle{\mu} \in \widehat{H} / \Z} M \tangle{\mu},
\]

\noindent where given $\mu \in \widehat{H}$, $\tangle{\mu} := \Z * \mu$
and $M \tangle{\mu} := \bigoplus_{n \in \Z} M_{n * \mu}$.
Thus $M$ has a finite filtration, each of whose subquotients is either a
quotient of a Verma module, or else a finite-dimensional weight module
$N$ such that $\wt N \subset \widehat{H} \setminus \hfree$. In
particular, $\calo$ is finite length if and only if every Verma module
has finite length.
\end{prop}

The next result is very useful in determining the structure of modules in
$\calo$.

\begin{prop}\label{Pext}
Given $\lambda,\mu \in \hfree$, $M \in \Ext^1_\calo(L(\mu),L(\lambda))$
is a non-split extension, if and only if there exists $0 \neq n \in \Z$
such that $\mu = n * \lambda$ and $M(\max(\lambda,\mu))$ surjects onto
one of $M$ and $F(M)$. In particular, the following are equivalent:
\begin{enumerate}
\item $\Ext^1_\calo(L(\mu),L(\lambda)) \neq 0$.
\item $\dim \Ext^1_\calo(L(\mu),L(\lambda)) = 1$.
\item $\Z * \lambda = \Z * \mu$ and $M(\min(\lambda,\mu))$ is the unique
maximal submodule in $M(\max(\lambda,\mu))$.
\end{enumerate}

\noindent In particular, $[\lambda] = [\mu]$.
\end{prop}

The following result uses Proposition \ref{Psmith1} to analyze Verma
modules in detail in Category $\calo$, and provides motivation for
considering the sets $[\lambda]$ used in Definition \ref{Delements} and
Theorem \ref{Tkoszul}.

\begin{prop}\label{Pverma}
Fix any triangular GWA $A$ and $\mu \in \hfree$. Then $M(\mu)$ is a
uniserial module, with unique composition series:
\[
M(\mu) \supset M((-n_1) * \mu) \supset M((-n_2) * \mu) \supset \cdots,
\]

\noindent where $0 < n_1 \leq n_2 \leq \cdots$ comprise the set $\{ n
\geq 1 : \mu(\z_n) = 0 \}$. Thus $\calo$ is finite length if and only if
$[\mu] \cap (-\Z_+ * \mu)$ is finite for every $\mu \in \hfree$.
Moreover, the following are equivalent, given $n \in \Z_+$ and $\mu \in
\hfree$:
\begin{enumerate}
\item The multiplicity $[M(n * \mu) : L(\mu)]$ is nonzero.
\item $[M(n * \mu) : L(\mu)] = 1$.
\item $M(\mu) \hookrightarrow M(n * \mu)$.
\item $(n * \mu)(\z_n) = 0$.
\item $\mu(\z_{-n}) = 0$.
\end{enumerate}
\end{prop}

\noindent In particular, every submodule of a Verma module is a Verma
module. Moreover, the result also provides a ``BGG resolution'' of every
simple object $L(\mu)$: $M(\mu) = L(\mu)$ if $\dim L(\mu) = \infty$;
otherwise given $n_1$ as in Proposition \ref{Pverma}, we have
$0 \to M((-n_1)*\mu) \to M(\mu) \to L(\mu) \to 0$.

\begin{remark}
Observe that the restricted dual of every finite-dimensional highest
weight module $M(\mu) / M((-n_r) * \mu)$ (with notation as in Proposition
\ref{Pverma}, and some $r \geq 0$ such that $0 < n_r < \infty$) is a
\textit{lowest weight module}, generated by its lowest weight vector of
weight $(1-n_1) * \mu$. In particular for $r=1$, the simple
finite-dimensional module $L(\mu)$ is both a highest weight module and a
lowest weight module, akin to semisimple Lie algebras.
\end{remark}

\comment{
Our last result in this section collects together additional useful facts
about Verma modules. To state this result, recall that the {\em
Shapovalov form} $\tangle{-,-}$ is defined to be an $H$-valued bilinear
form on $A$, given by $\tangle{a,b} = \xi(i(a)b)$, where $\xi$ is the
\textit{Harish-Chandra projection} $: A = H \oplus (dA+Au) \to H$, and
$i$ is the fixed anti-involution on $A$ that fixes $H$ and interchanges
$u,d$.

\begin{prop}\label{Psmith2}
Suppose $A = \cals(H,\theta,z_0,z_1)$ as above. Then the Shapovalov form
(of every Verma module -- equivalently,) of $\F[d]$ is given by
$\tangle{d^m, d^n} = \delta_{m,n} \prod_{j=0}^{n-1} \z_j$ for all $m, n
\geq 0$. Moreover, given $\mu \in \hfree$, the following are equivalent:
\begin{enumerate}
\item There exists $\nu < \mu$ (i.e., $\nu = (-n) * \mu$ for some $n>0$)
such that\hfill\break $\dim \Hom_\calo(M(\nu),M(\mu)) = 1$.

\item There exists $\nu < \mu$ such that $\Hom_\calo(M(\nu),M(\mu)) \neq
0$.

\item There exists $\nu < \mu$ such that $M(\nu) \hookrightarrow M(\mu)$.

\item There exists $\nu < \mu$ such that $L(\nu)$ is a subquotient of
$M(\mu)$.

\item $M(\mu)$ is not simple.

\item There exists $n \in \N$, such that $\mu(\z_n) = 0$.

\item $L(\mu)$ is finite-dimensional.
\end{enumerate}
\end{prop}

\begin{proof}
The first assertion follows by repeatedly applying Equation
\eqref{Ecomp}. The equivalences in the second part are now shown by
standard arguments and using Proposition \ref{Pverma}.
\end{proof}
}

\subsection{Category $\calo$: projectives, blocks, and highest weight
categories}\label{SbggO}

The next step in proving Theorem \ref{Tkoszul} is to construct projective
modules in $\calo$. Note from the equivalences in Proposition
\ref{Pverma} that $[\lambda]$ is related to partitioning the set of
weights $\hfree$ (and hence, Category $\calo$) into blocks. Thus it is an
analogous notion to that of \textit{linkage} for semisimple Lie algebras,
as well as to \textit{Condition (S3)} in the axiomatic framework studied
in \cite{Kh2}. Note that these three notions coincide when $A =
\cals(H,\theta,z_0,z_1) = U(\lie{sl}_2)$.

We now recall additional standard constructions from \cite{Kh2}, for
which the following notation is required.

\begin{defn}
Set $A := \cals(H,\theta,z_0,z_1)$ as above.
Given $\lambda \in \widehat{H}$ and an integer $l \geq 0$, define
$P(\lambda,l) := A / (A u^l + A \cdot \ker \lambda)$, and
$\calo(\lambda,l) \subset \calo$ to be the full subcategory of all $M \in
\calo$ such that $u^l M_\lambda = 0$.
We also say that an object $X$ in $\calo$ has a \textit{(dual) Verma
flag} if $X$ has finite filtration in $\calo$ whose subquotients are
(restricted duals of) Verma modules.
\end{defn}

We now have the following standard results in the study of Category
$\calo$; we avoid the proofs as these results are shown in greater
generality in \cite{Kh2}.
\begin{enumerate}
\item $\calo_\N$ is a direct sum of \textit{blocks}:
\begin{equation}\label{Eblock}
\calo_\N = \bigoplus_{\mu \in (\widehat{H} \setminus \hfree)/\Z}
\calo_\N(\Z * \mu) \oplus \bigoplus_{[\lambda] \subset \hfree}
\calo_\N[\lambda].
\end{equation}

\noindent All morphisms and extensions between objects of distinct blocks
(i.e., distinct summands) are zero.

\item If $T \subset \widehat{H}$ is finite, then $\calo(T)$ is a finite
length, abelian $\F$-category.

\item For all finite $T \subset \hfree$ and all $\lambda \in
\widehat{H}$, there exists $l \geq 0$ such that $\calo(T) \subset
\calo(\lambda,l)$.

\item For all $l > 0$ and $\lambda \in \hfree$, $P(\lambda,l)$ is a
projective module in $\calo(\lambda,l)$.

\item Suppose $\calo$ is finite length and $[\lambda]$ is finite for some
$\lambda \in \hfree$. Fix $l \geq 0$ such that $\calo[\lambda] \subset
\calo(\lambda, l)$, and let $P(\lambda)$ be the direct summand
corresponding to the block $[\lambda]$, in the decomposition of
$P(\lambda,l)$ according to \eqref{Eblock}.
Then:
\begin{itemize}
\item $P(\lambda)$ is the projective cover of $L(\lambda)$ in
$\calo[\lambda]$ (and hence in $\calo$).
\item $P(\lambda)$ has a Verma flag, with Verma subquotients of the form
$M(\mu)$ for $\mu \in [\lambda]$.
\item $\calo[\lambda]$ has enough projectives and injectives.
\end{itemize}

\item If $\calo$ is finite length and $[\lambda]$ is finite for $\lambda
\in \hfree$, then $\calo[\lambda]$ is equivalent to finite-dimensional
(left) modules over a finite-dimensional quasi-hereditary algebra $\ba$.
In particular, it is a highest weight category (see \cite{CPS1}, as well
as \cite[Section 3]{Kh2} for further consequences) that satisfies
\textit{BGG Reciprocity}:
\[
[P(\mu) : M(\nu)] = [M(\nu) : L(\mu)], \quad \forall \mu \in [\lambda], \
\nu \in \hfree.
\]
\end{enumerate}

The algebra $\ba$ is obtained as follows: if $[\lambda] = \{ \lambda_1 <
\lambda_2 < \cdots < \lambda_n \} \subset \hfree$, and $P(\lambda_i)$ is
the projective cover of $L(\lambda_i)$ in $\calo[\lambda]$ (and hence in
$\calo$), then $\ba = \End_\calo(\bigoplus_{i=1}^n P(\lambda_i)^{\oplus
r_i})^{op}$, for any choice of positive integers $r_i$. Up to Morita
equivalence, we may choose $r_i = 1$ for all $i$.

We conclude this section by observing that the above standard facts prove
the first part of Theorem \ref{Tkoszul} except for the algebra
$\ba$ being graded Koszul. It is the goal of the following sections
to prove the remaining, more involved homological assertions in Theorems
\ref{Tkoszul}--\ref{Tproj}.

\section{Projectives and resolutions}\label{Sproj}

In this section and the next, we carry out the technical heart of the
computations needed to show the main results in this paper. We end the
section by proving Theorem \ref{Thom}.

The remainder of this paper operates under the following assumptions.

\begin{stand}\label{Stand3}
Henceforth assume that Assumption \ref{Stand2} holds, $\calo$ is finite
length, and the block $[\lambda]$ is finite for some $\lambda \in
\hfree$.
\end{stand}

We also set some \textbf{notation}. Enumerate the weights in the block as
follows: $[\lambda] = \{ \lambda_1 < \lambda_2 < \cdots < \lambda_n \}
\subset \hfree$. Given $1 \leq i, j \leq n$, we abuse notation and define
$\lambda_i - \lambda_j$ to be the (unique) integer $n$ such that $n *
\lambda_j = \lambda_i$. Recall also the notation of $L_i, M_i, P_i$ as in
Equation \eqref{Enotation}.

We begin by ascertaining the structure of every indecomposable projective
object $P(\lambda)$ in $\calo[\lambda]$.

\begin{prop}\label{Pproj}
We work in $\calo[\lambda]$. For all $1 \leq i \leq n$, $M_i$ has a
finite filtration
\[
M_i \supset M_{i-1} \supset \cdots \supset M_1 \supset 0,
\]

\noindent with successive subquotients $L_j$ for $1 \leq j \leq i$.
Dually, every $P_i$ has a finite filtration
\[
P_i \supset P_{i+1} \supset \cdots \supset P_n \supset 0,
\]

\noindent with successive subquotients $M_j$ for $i \leq j \leq n$.
Moreover, $\calo[\lambda] \subset \calo(\lambda_i, \lambda_n - \lambda_i
+ 1)$ for all $i$.
\end{prop}

\begin{proof}
The filtration of each Verma module $M_i$ is discussed in Proposition
\ref{Pverma}. Next, $\calo[\lambda] \subset \calo(\lambda_i, \lambda_n -
\lambda_i + 1)$ for all $i$ by \cite[Section 3]{Kh2}. Therefore the
$[\lambda]$-summand of $P(\lambda_i, \lambda_n - \lambda_i + 1)$ is
precisely $P_i$, from above.
We now consider the structure of $P(\lambda_i, l)$ for any $l>0$ and $1
\leq i \leq n$: if $p_{\lambda_i}$ is the image of $1$ in $P(\lambda_i,
n)$, then
\[
0 = N_0 \subset N_1 \subset N_2 \subset \cdots \subset N_l = P(\lambda_i,
l) = A p_{\lambda_i},
\]

\noindent where $N_k := A u^{l-k} p_{\lambda_i}$. It is then easy to
verify by comparing formal characters that $N_k = P((l-k) *
\lambda_i,k)$, and that $N_k / N_{k-1} \cong M((l-k) * \lambda_i)\
\forall 1 \leq k \leq l$.

Now set $l := \lambda_n - \lambda_i + 1$. Also let $N_k[\lambda]$ denote
the $[\lambda]$-component of $N_k$ under the decomposition \eqref{ES3}.
Then $N_k[\lambda] = N_{k-1}[\lambda]$ unless that particular subquotient
-- namely, $N_k / N_{k-1} = M((l-k) * \lambda_i)$ -- equals $M_j$ for
some $j \geq i$. Otherwise if $N_k / N_{k-1} \cong M_j$, then
$N_k[\lambda] / N_{k-1}[\lambda] = N_k / N_{k-1} \cong M_j$, by repeated
applications of Proposition \ref{Pext} in the abelian category $\calo =
\calo_\N$.

Since $P(\lambda_i) = P_i$ is the $[\lambda]$-summand of $P(\lambda_i,
\lambda_n - \lambda_i + 1)$ for all $i$, we thus obtain the following
commuting sequence, by considering only those $N_k$'s in $P(\lambda_i,
\lambda_n - \lambda_i + 1)$ which correspond to some $\lambda_j$ for
$j>i$:
\[
 \begin{CD}
  0  @>>>    P_n @>>> P_{n-1} @>>> \cdots @>>> P_i\\
  @VVV       @VVV     @VVV         @.         @VVV\\
  0  @>>>    P(\lambda_n, 1)   @>>> P(\lambda_{n-1}, \lambda_n -
  \lambda_{n-1} + 1)      @>>> \cdots @>>> P(\lambda_i, \lambda_n -
  \lambda_i + 1)\\
 \end{CD}
\]

\noindent Over here, all arrows are inclusions, and the subquotients in
the top row are Verma modules $M_j$ for $i \leq j \leq n$. Moreover, each
vertical arrow represents the inclusion of the corresponding
$[\lambda]$-summand, which concludes the proof.
\end{proof}

\begin{remark}\label{Rgen}
In fact, if $p_i$ is the image of $1$ in $P(\lambda_i, \lambda_n -
\lambda_i + 1)$ (and hence the generator of its quotient $P_i$ as well),
then it is easy to check that $u^{\lambda_{i+1} - \lambda_i} p_i$ is the
image of the generator $\overline{1}$ in $P(\lambda_{i+1}, \lambda_n -
\lambda_{i+1} + 1)$.
Also note that if we reverse both vertical arrows or the right-hand
vertical arrow in any commuting square in the diagram (by the
corresponding projection maps onto the $[\lambda]$-summands), then we
still obtain a commuting square.
\end{remark}

The following result provides a projective resolution in $\calo$ of every
highest weight module.

\begin{prop}\label{Presol}
Suppose $0 < s < r \leq n$. Then the following is a projective resolution
of the highest weight module $M_r / M_s$ in $\calo$:
\begin{equation}
0 \to P_{s+1} \to P_s \oplus P_{r+1} \to P_r \to M_r / M_s \to 0,
\end{equation}

\noindent with the understanding that $P_{n+1} = 0$. If $0 = s < r \leq
n$, then the Verma module $M_r$ has a projective resolution:
\begin{equation}
0 \to P_{r+1} \to P_r \to M_r \to 0, \quad \forall 1 \leq r \leq n.
\end{equation}
\end{prop}

\begin{proof}
We begin with the following observation:
\begin{equation}
\Hom_\calo(P_i, L_j) = \Hom_\calo(M_i, L_j) = \Hom_\calo(L_i, L_j) =
\delta_{i,j} \F, \quad \forall 1 \leq i,j \leq n.
\end{equation}

Now note first that the theorem holds for all Verma modules $M_r$ by
Proposition \ref{Pproj}. Thus, for the remainder of the proof we fix $0 <
s < r \leq n$.
Suppose $0 \to K \to P_r \to M_r / M_s \to 0$;
then the kernel $K$ equals the lift to $P_r$ of $M_s \subset M_r =
P_r / P_{r+1}$. In other words,
\[
0 \to P_{r+1} \mapdef{\iota} K \mapdef{\pi} M_s \to 0.
\]

\noindent Next, the surjection $P_s \twoheadrightarrow M_s$ factors
through a map $: P_s \mapdef{\varphi} K \mapdef{\pi} M_s$, since $P_s$ is
projective. Thus, define a map $f : P_{r+1} \oplus P_s \to K$ via:
$f(m,n) := \varphi(n) - \iota(m)$. This is a morphism in
$\calo[\lambda]$, and given $k \in K$, choose $n \in P_s$ such that
$\varphi(n) - k \in P_{r+1}$. Then $k = f(\varphi(n) - k, n)$, whence $f$
is a surjection.

It remains to compute the kernel of $f$, which equals $\{ (\varphi(n),n)
: \varphi(n) \in P_{r+1} \}$. Now $\varphi(n) \in P_{r+1}$ if and only if
$\pi(\varphi(n)) = 0$, if and only if (from the above factoring of the
surjection) $n$ is in the kernel of $P_s \twoheadrightarrow M_s$, if and
only if $n \in P_{s+1}$ (by Proposition \ref{Pproj}). But then the map
$\psi : P_{s+1} \to \ker(f)$, given by $\psi(n) := (\varphi(n),n)$, is an
$A$-module isomorphism, since $\varphi$ is an $A$-module map. This yields
the required projective resolution of $M_r / M_s$ in $\calo$.
\end{proof}

An easy consequence of Proposition \ref{Presol} is that we can now
compute all $\Ext$-groups between simple objects in the block
$\calo[\lambda]$ (and hence, in $\calo$), as shown presently. Indeed, the
\textit{Hilbert matrix} of the block $\calo[\lambda]$ is defined to be 
\begin{equation}\label{Ehilbert}
H(E(\ba),t)_{i,j} := \sum_{l \geq 0} t^l \dim \Ext^l_\calo(L_i, L_j).
\end{equation}

\noindent The matrix $H(E(\ba),t)$ encodes homological information
in the block $\calo[\lambda]$, and can now be computed explicitly:

\begin{cor}\label{Cext}
For all $1 \leq i,j \leq n$ and $l \geq 0$, $\Ext^l_\calo(L_i, L_j)$
satisfies the formula stated in Theorem \ref{Thom}.
In particular, the Hilbert matrix of $\calo[\lambda]$ (or of $E(\ba)$) is
the following symmetric tridiagonal $n \times n$ matrix with determinant
$1$:
\begin{equation}
H(E(\ba),t) =
  \begin{pmatrix}
    1 & t & 0 & \cdots & 0 & 0\\
    t & 1+t^2 & t & \cdots & 0 & 0\\
    0 & t & 1+t^2 & \cdots & 0 & 0\\
    \vdots & \vdots & \vdots & \ddots & \vdots & \vdots\\
    0 & 0 & 0 & \cdots & 1+t^2 & t\\
    0 & 0 & 0 & \cdots & t & 1+t^2
  \end{pmatrix}.
\end{equation}
\end{cor}

\noindent The corollary follows easily from Propositions \ref{Pext} and
\ref{Presol}; for sake of brevity, we do not elaborate further here, as
the steps are similar to those in proving Theorem \ref{Tproj}(1)
below. Note that the determinant of the given $n \times n$ matrix can be
computed by induction on $n$ and expanding along the last row.

We now prove additional homological properties of the block
$\calo[\lambda]$. Note by Proposition \ref{Pproj} that inside each module
$P_j / P_k$ for $1 \leq j < k \leq n+1$, sits a copy of the Verma module
$M_{k-1}$. Thus, $\dim \Hom_\calo(M_i, P_j / P_k) \geq {\bf 1}(i < k)$
for $1 \leq i \leq n$. We now show that this inequality is actually an
equality -- namely, that inside each projective cover, there is at most
one maximal vector of each possible weight $\lambda_i$. In particular,
this helps prove one of our main results.

\begin{proof}[Proof of Theorem \ref{Thom}]
The assertions prior to Equation \eqref{Eext} were shown in Proposition
\ref{Pproj} and Corollary \ref{Cext}. We next claim that
\begin{equation}\label{Emult}
\dim \Hom_\calo(P_i, M) = [M : L_i], \ \forall M \in \calo, \ 1 \leq i
\leq n.
\end{equation}

\noindent The equation holds because the functions $\dim \Hom_\calo(P_i,
-)$ and $[- : L_i]$ are additive on short exact sequences, and both equal
$\delta_{ij}$ when evaluated at a simple object $L_j$.

The heart of the proof involves showing the first assertion in Equation
\eqref{Eext}. For this, we first \textbf{claim} that $\dim
\Hom_\calo(M(\mu), P(\lambda,l)) \leq 1$ for all $\lambda,\mu \in \hfree$
and all integers $l \geq 1$. The claim is obvious if $\lambda \notin \Z *
\mu$. Now suppose $\mu = n_0 * \lambda$ for some $n_0 \in \Z$, and define
$\max(\lambda,\mu)$ to be $\mu$ if $n_0 \geq 0$, and $\lambda$ otherwise.
Let $m_\lambda, m_\mu \in \Z_+$ denote the unique integers such that
$m_\lambda * \lambda = m_\mu * \mu = \max(\lambda,\mu)$; then $n_0 =
m_\lambda - m_\mu$. Note that if $\Hom_\calo(M(\mu), P(\lambda,l)) \neq
0$ then $\mu < l * \lambda$. Now verify that $P(\lambda,l)_\mu$ is
spanned by
\[
\{ d^{m_\mu + i} u^{m_\lambda + i} \overline{1}_\lambda : 0 \leq i \leq l
- 1 - \max(\lambda,\mu) \},
\]

\noindent where  $\overline{1}_\lambda$ is the generating vector in (the
definition of) $P(\lambda,l)_\lambda$. (We use here that $d^{m_\mu + i}
u^{m_\lambda + i}$ kills $\overline{1}_\lambda$ if $i \geq l -
\max(\lambda,\mu)$.) Use Proposition \ref{Pakaki} to conclude that $V :=
P(\lambda,l)_\mu$ is a finite-dimensional quotient of $A[m_\lambda -
m_\mu]$, and hence a finite-dimensional $\F[du]$-module. It follows that
$\dim \coker(du|_V) = \dim \ker(du|_V) \leq 1$. Now $\ker(u|_V) \subset
\ker (du|_V)$, so $P(\lambda,l)_\mu$ has at most one maximal vector (up
to scalar multiples). This proves the claim.

The next step is to note that since $M_i \hookrightarrow M_n = P_n
\hookrightarrow P_j$ for all $1 \leq j \leq n$, hence by the claim,
\[
1 \leq \dim \Hom_\calo(M_i, P_j) \leq \dim \Hom_\calo(M_i, P(\lambda_j,
\lambda_n - \lambda_j + 1)) \leq 1.
\]

\noindent Thus all inequalities are equalities, and $\dim
\Hom_\calo(M_i,P_j) = 1$ for all $i,j$.

We now compute $\Ext^1_\calo(M_i,P_j)$. Note that, given any object $X$
in $\calo[\lambda]$, applying the functor $\Hom_\calo(-, X)$ to the short
exact sequence $0 \to P_{i+1} \to P_i \to M_i \to 0$ yields the long
exact sequence:
\[
0 \to \Hom_\calo(M_i, X) \to \Hom_\calo(P_i, X) \to \Hom_\calo(P_{i+1},
X) \to \Ext_\calo^1(M_i, X) \to \Ext_\calo^1(P_i, X) \to \cdots
\]

\noindent The last term is zero since $P_i$ is projective. (Also note
that all higher $\Ext$-groups are zero.) Thus the Euler characteristic of
the terms listed above is zero, which yields via
\eqref{Emult}:
\begin{equation}\label{Eakaki}
\dim \Ext_\calo^1(M_i,X) = [X : L_{i+1}] - [X : L_i] + \dim
\Hom_\calo(M_i,X), \quad \forall 1 \leq i \leq n, \ X \in
\calo[\lambda].
\end{equation}

\noindent Now apply Equation \eqref{Eakaki} for $X = P_j$; then there are
two cases. First, if $1 \leq i < j$, then $[P_j : L_i] = [P_j : L_{i+1}]$
by Proposition \ref{Pproj}, so by Equation \eqref{Eakaki},
\[
\dim \Ext_\calo^1(M_i, P_j) = \dim \Hom_\calo(M_i, P_j) = 1.
\]

\noindent Similarly, if $i \geq j$, then  $[P_j : L_i] = [P_j : L_{i+1}]
+ 1$, whence $\dim \Ext_\calo^1(M_i, P_j) = 0$.

We now show the results for $\Hom_\calo(M_i, P_j/P_k)$ and
$\Ext^1_\calo(M_i, P_j/P_k)$ simultaneously. We assume below that $k \in
(j,n+1)$, since the $k=n+1$ case follows from the above analysis. First
suppose that $i \geq k$, and apply $\Hom_\calo(M_i,-)$ to the short exact
sequence $0 \to P_k \to P_j \to P_j/P_k \to 0$, to obtain:
\[
0 \to \Hom_\calo(M_i, P_k) \to \Hom_\calo(M_i, P_j) \to \Hom_\calo(M_i,
P_j/P_k) \to \Ext^1_\calo(M_i,P_k) \to \cdots
\]

\noindent Since the last term is zero from above, computing the Euler
characteristic via the above analysis yields: $\Hom_\calo(M_i, P_j/P_k) =
0$. Now use Equation \eqref{Eakaki} with $X = P_j/P_k$ to conclude that
$\Ext^1_\calo(M_i, P_j/P_k) = 0$ for all $1 \leq j < k \leq i$.

We next carry out a similar analysis for $i < k$, applying
$\Hom_\calo(M_i,-)$ to the short exact sequence $0 \to P_k \to P_j \to
P_j/P_k \to 0$, to obtain:
\begin{align}\label{Elong}
0 \to \Hom_\calo(M_i, P_k) \to & \Hom_\calo(M_i, P_j) \to \Hom_\calo(M_i,
P_j/P_k)\notag\\
\to & \Ext^1_\calo(M_i,P_k) \to \Ext^1_\calo(M_i,P_j) \to \cdots
\end{align}

\noindent There are now two sub-cases:
\begin{enumerate}
\item First suppose that $j \leq i < k$. Since the last term in
\eqref{Elong} is zero from above, computing the Euler characteristic via
the above analysis yields: $\Hom_\calo(M_i, P_j/P_k) = 1$. Now use
Equation \eqref{Eakaki} with $X = P_j/P_k$ to get: $\Ext^1_\calo(M_i,
P_j/P_k) = 0$ for $1 \leq j \leq i < k$.

\item If instead $i < j$, then first note that $\Hom_\calo(M_i,
P_j/P_k) \neq 0$. Additionally, in the long exact sequence \eqref{Elong},
the first two terms are one-dimensional, whence
\[
0 \neq \Hom_\calo(M_i, P_j/P_k) = \ker \left( \Ext^1_\calo(M_i,P_k) \to
\Ext^1_\calo(M_i,P_j) \right).
\]
\end{enumerate}

\noindent Now since $\Ext^1_\calo(M_i,P_k)$ is one-dimensional, it
follows that so is $\Hom_\calo(M_i, P_j/P_k)$. Finally, apply Equation
\eqref{Eakaki} with $X = P_j/P_k$ to get: $\dim \Ext^1_\calo(M_i,P_j/P_k)
= 1$ if $1 \leq i < j < k$.\medskip

Thus we have shown the first assertion in \eqref{Eext}. The second
assertion is clear for $l \geq 2$; in fact, $\Ext^l_\calo(P_j/P_k, X) =
0$ for all $X \in \calo[\lambda]$ and $l \geq 2$.

It remains to show the second assertion in \eqref{Eext} for $l=0,1$.
First note that if $\varphi : P_j/P_k \to M_r/M_s$ is nonzero, then the
generating vector $\overline{1}_{\lambda_j} \in P_j/P_k$ maps to a
nonzero weight vector in $M_r/M_s$ of weight $\lambda_j$. Therefore
$\lambda_s < \lambda_j \leq \lambda_r$, i.e., $s < j \leq r$. Moreover,
any such nonzero homomorphism is unique since $\dim (M_r/M_s)_{\lambda_j}
\leq 1$. This shows that $\dim \Hom_\calo(P_j/P_k, M_r/M_s) = {\bf 1}(s <
j \leq r)$, as claimed. Finally, apply $\Hom_\calo(-, M_r/M_s)$ to the
short exact sequence
\[
0 \to P_k \to P_j \to P_j / P_k \to 0
\]

\noindent to obtain the long exact sequence:
\begin{align*}
0 \to \Hom_\calo(P_j / P_k, M_r/M_s) & \to \Hom_\calo(P_j, M_r/M_s) \to
\Hom_\calo(P_k, M_r/M_s)\\
& \to \Ext^1_\calo(P_j / P_k, M_r/M_s) \to \Ext^1_\calo(P_j, M_r/M_s) \to
\cdots
\end{align*}

\noindent Since the last term is zero, and the Euler characteristic of
this terms of the sequence displayed above is zero as well, we compute
using Equation \eqref{Emult}:
\begin{align*}
\dim \Ext^1_\calo(P_j/P_k, M_r/M_s) = &\ [M_r/M_s : L_k] - [M_r/M_s :
L_j] + {\bf 1}(s < j \leq r)\\
= &\ {\bf 1}(s < k \leq r) - {\bf 1}(s < j \leq r) + {\bf 1}(s < j \leq
r) = {\bf 1}(s < k \leq r),
\end{align*}

\noindent as claimed.
\end{proof}

For completeness, we also compute the morphisms between highest weight
modules and quotients of Verma modules.

\begin{prop}\label{Phom}
Fix integers $0 \leq s < r \leq n$ and $0 \leq j < k \leq n+1$. Then,
\begin{alignat}{3}\label{Eprojext}
\dim \Hom_\calo(M_r / M_s, M_k / M_j) = &\ \mathbf{1}(s \leq j < r \leq
k), \qquad & \text{if } k \leq n,\\
\dim \Hom_\calo(M_r / M_s, P_j / P_k) = &\ \delta_{s,0} \mathbf{1}(r <
k), & \text{if } j \geq 1.\notag
\end{alignat}
\end{prop}

\begin{proof}
We will use the following consequence of Equation \eqref{Emult} without
further reference:
\begin{equation}
\dim \Hom_\calo(P_i, M_k / M_j) = [M_k / M_j : L_i]
= \mathbf{1}(j < i \leq k).
\end{equation}

We now show the first assertion. If $s=0$, then $\Hom_\calo(M_r, M_k /
M_j) = \Hom_\calo(P_r, M_k / M_j)$ has dimension $\mathbf{1}(j < r \leq
k)$. If instead $s>0$, then $\Hom_\calo(M_r / M_s, M_k / M_j)$ consists
of precisely the maps $: M_r \to M_k / M_j$ such that the image of $M_s$
is killed. By the above analysis, this happens if and only if $j < r \leq
k$ and $s \leq j$, proving the first assertion.

Now note that the $s=0$ case of the second assertion was shown in Theorem
\ref{Thom}. If instead $s>0$, then every morphism $: M_r / M_s \to P_j /
P_k$ gives rise to a morphism $: M_r \to P_j / P_k$. By Theorem
\ref{Thom}, no such map kills $M_s \subset M_r$, so $\Hom_\calo(M_r /
M_s, P_j / P_k) = 0$ as claimed.
\end{proof}

\section{Tilting modules and submodules of projective
modules}\label{Stilting}

The goal of this section is to prove Theorem \ref{Tproj}, which
classifies all the tilting modules as well as submodules of quotients of
projectives $P_r / P_s$ in the block $\calo[\lambda]$. A crucial
ingredient in this analysis is the study of maps between quotients of
projective objects in the block $\calo[\lambda]$. This is the focus of
the next subsection.

\subsection{Graded maps between quotients of projective modules}

Recall that in order to prove Theorem \ref{Tkoszul}, we need to study the
algebra
\[
\ba = \End_\calo(\bp), \quad \text{where} \quad \bp = \bigoplus_{1 \leq r
\leq n} P_r.
\]

\noindent Our aim in this subsection is to first study the larger algebra
\[
\widetilde{\ba} = \End_\calo(\widetilde{\bp}), \quad \text{where} \quad
\widetilde{\bp} = \bigoplus_{1 \leq r < s \leq n+1} P_r / P_s.
\]

\noindent The first goal is to show that $\widetilde{\ba}$ is a
finite-dimensional, $\Z_+$-graded $\F$-algebra with a distinguished
basis, a subset of which spans the subalgebra $\ba$. We begin by
considering one such family of maps.

\begin{prop}\label{Pses}
Given integers $1 \leq r \leq s \leq n$, we have the following short
exact sequence in the block $\calo[\lambda]$:
\begin{equation}\label{Eses}
0 \to P_r / P_s \mapdef{f_{r,s}^{++}} P_{r+1} / P_{s+1} \to F(M_s / M_r)
\to 0,
\end{equation}
where $F$ is the restricted duality functor defined in Remark
\ref{Rverma}(6).
\end{prop}

\begin{proof}
We begin by proving the claim that there exists an injection
$f_{r,s}^{++} : P_r / P_s \hookrightarrow P_{r+1} / P_{s+1}$.
The proof is by reverse induction on $r \in [1,s]$. For $r=s, s-1$, the
assertion is immediate since $P_{s-1} / P_s \cong M_{s-1}$ for all $s$.
Now suppose the assertion holds for $r+1 \leq s$. We then have
\begin{equation}\label{Eproof21}
0 \to P_{r+1}/P_s \to P_r/P_s \to M_r \to 0,
\end{equation}

\noindent and $f_{r+1,s}^{++} : P_{r+1}/P_s \hookrightarrow P_{r+2} /
P_{s+1}$. If we push-forward \eqref{Eproof21} by $f_{r+1,s}^{++}$ we get
an exact sequence
\begin{equation}\label{Eproof22}
0 \to P_{r+2} / P_{s+1} \to N \to M_r \to 0.
\end{equation}

\noindent We now make the sub-claim that the extension $N$ in
\eqref{Eproof22} is the same as the submodule $N'$ of $P_{r+1} /
P_{s+1}$, given by:
\begin{equation}\label{Eproof23}
0 \to P_{r+2} / P_{s+1} \to N' \to M_r \to 0,
\end{equation}

\noindent where $N'$ is the pre-image of $M_r$ under $P_{r+1}
\twoheadrightarrow M_{r+1}$. In order to prove the sub-claim, it suffices
to prove the following facts:
\begin{enumerate}
\item The short exact sequences in \eqref{Eproof21} and \eqref{Eproof23}
are non-split;

\item $f_{r+1,s}^{++}$ induces an isomorphism $: \Ext^1_\calo(M_r,
P_{r+1}/P_s) \to \Ext^1_\calo(M_r, P_{r+2} / P_{s+1})$;

\item $\dim \Ext^1_\calo(M_r, P_{r+2} / P_{s+1}) = 1$. (This was already
shown in Theorem \ref{Thom}.)
\end{enumerate}

We first show (1). By Theorem \ref{Thom}, the map $: \Hom_\calo(M_r,
P_{r+1}/P_s) \to \Hom_\calo(M_r, P_r/P_s)$ induced by post-composing with
inclusion, is a nonzero map between one-dimensional vector spaces, hence
is an isomorphism. It follows that \eqref{Eproof21} is non-split.
Similarly, using Theorem \ref{Thom} and that $\Hom_\calo(M_r, P_{r+2} /
P_{s+1}) = \Hom_\calo(M_r, P_{r+1} / P_{s+1}) = \Hom_\calo(M_r, N')$,
shows that \eqref{Eproof23} is also non-split.

To show (2), it suffices to show that
\[
\Hom_\calo(M_r, \coker(f_{r+1,s}^{++})) =
\Ext^1_\calo(M_r, \coker(f_{r+1,s}^{++})) = 0,
\]

\noindent by using an appropriate long exact sequence obtained from the
inclusion $f_{r+1,s}^{++}$. Now note by Equation \eqref{Emult} that the
Jordan-Holder factors of $\coker(f_{r+1,s}^{++})$ are precisely one copy
of $L_j$ for $r+2 \leq j \leq s+1$. Thus to show (2) it suffices to prove
that $\Hom_\calo(M_r, L_j)=0=\Ext^1_\calo(M_r, L_j)$ for $j \geq r+2$.
But this follows immediately from the projective resolution of $M_r$.
This concludes the proof of the claim that $f_{r,s}^{++} : P_r / P_s \to
P_{r+1} / P_{s+1}$ is an injection.

To complete the proof of Equation \eqref{Eses}, let $V$ denote the
cokernel of the inclusion $f_{r,s}^{++} : P_r / P_s \hookrightarrow
P_{r+1} / P_{s+1}$. We first show the sub-claim that the vectors
\begin{equation}\label{Edual-verma}
\{ v_t := u^{\lambda_t - \lambda_{r+1}} 1_{P_{r+1} / P_{s+1}} : r+1 \leq
t \leq s \}
\end{equation}

\noindent are not contained in $P_r/P_s$, and hence have nonzero images
in $V$. The proof is by contradiction; thus, suppose for some integer $t
\in [r+1,s]$ that $v_t = u^{\lambda_t - \lambda_{r+1}} 1_{P_{r+1} /
P_{s+1}} \in P_r / P_s$. By the proof of Proposition \ref{Pproj}, $A v_t
\cong P_t / P_{s+1}$ then embeds into $P_r / P_s$. But then Theorem
\ref{Thom} would imply:
\[
1 = \dim \Hom_\calo(M_s, P_t / P_{s+1}) = \dim \Hom_\calo(M_s, A v_t)
\leq \dim \Hom_\calo(M_s, P_r / P_s) = 0,
\]

\noindent which is impossible, and hence shows the sub-claim.

Now consider the module $V = \coker(f_{r,s}^{++})$, which is a weight
module containing the weight vectors $v_t$, and with composition factors
$\{ L_t : r+1 \leq t \leq s \}$. This implies that $V$ is a
finite-dimensional, lowest weight module with specified formal character.
Now $V$ is easily verified to be the dual of the highest weight module
$M_s / M_r$, which completes the proof.
\end{proof}

Our next goal is to produce a distinguished $\Z_+$-graded basis of
$\widetilde{\ba}$. For this we first introduce the maps
\begin{equation}
f_{jk}^{++} : P_j / P_k \hookrightarrow P_{j+1} / P_{k+1}, \qquad
f_{jk}^{- \bullet} : P_j / P_k \hookrightarrow P_{j-1} / P_k, \qquad
f_{jk}^{\bullet -} : P_j / P_k \twoheadrightarrow P_j / P_{k-1}.
\end{equation}

\noindent Here, $f_{jk}^{++}$ was defined in Proposition \ref{Pses},
while the other two maps are canonically induced by the inclusion of
$P_j$ in $P_{j-1}$ for all $1 \leq j \leq n$, from Proposition
\ref{Pproj}. Now define for integers $1 \leq r < s \leq n+1$, $1 \leq j <
k \leq n+1$, and suitable $t > 0$:
\begin{equation}\label{Ephi}
\varphi_{(r,s), (j,k)}^{(t)} :=
\underbrace{f^{- \bullet}_{j+1,k} \circ \cdots \circ f^{-
\bullet}_{k-t,k}}_{k-j-t} \circ
\underbrace{f^{++}_{k-t-1,k-1} \circ \cdots \circ f^{++}_{r,r+t}}_{k-r-t}
\circ \underbrace{f^{\bullet -}_{r,r+t+1} \circ \cdots \circ f^{\bullet
-}_{r,s}}_{s-r-t}.
\end{equation}

\noindent Observe that Equation \eqref{Ephi} shows the maps
$\varphi_{(r,s), (j,k)}^{(t)}$ to be defined only for $1 \leq t \leq
\min(s-r, k-r, k-j)$. Our next result shows that the family of maps
\eqref{Ephi} provides the aforementioned graded basis of the algebra
$\widetilde{\ba}$.

\begin{prop}\label{Pprojmaps}
Setting as in Theorems \ref{Tkoszul} and \ref{Thom}.
\begin{enumerate}
\item Fix integers $1 \leq \{ r,s \} \leq j \leq k \leq n + 1$. Then the
image of the vector
\[
d^{\lambda_j - \lambda_s} u^{\lambda_j - \lambda_r} 1_{P_r/P_k} \in
P_r/P_k
\]

\noindent generates the submodule $P_s / P_{s+k-j}$ of $P_j / P_k
\hookrightarrow P_r / P_k$.

\item The maps $f_{jk}^{++}, f_{jk}^{- \bullet}, f_{jk}^{\bullet -}$
generate the $\F$-algebra $\widetilde{\ba} =
\End_\calo(\widetilde{\bp})$. Moreover, the maps
\[
\{ \varphi_{(r,s), (j,k)}^{(t)} : 1 \leq r < s \leq n+1, 1 \leq j < k
\leq n+1, 1 \leq t \leq \min(s-r, k-r, k-j) \}
\]

\noindent form a $\Z_+$-graded basis of $\widetilde{\ba}$.
Under this grading on $\widetilde{\ba}$,
\begin{equation}\label{Egrading2}
\deg f_{jk}^{++} = \deg f_{jk}^{- \bullet} = 1, \quad
\deg f_{jk}^{\bullet -} = 0, \quad
\deg \varphi_{(r,s), (j,k)}^{(t)} = 2(k-t) - r - j.
\end{equation}

\noindent Furthermore, if $1 \leq a < b \leq n+1$, then for all choices
of suitable $u,t$,
\begin{equation}\label{Egrading}
\varphi_{(j,k), (a,b)}^{(u)} \circ \varphi_{(r,s), (j,k)}^{(t)} =
{\bf 1}(u+t+j-k > 0) \varphi_{(r,s), (a,b)}^{(u+t+j-k)}.
\end{equation}

\item For all integers $1 \leq r < s \leq n+1$, the module $P_r / P_s$ is
indecomposable.
\end{enumerate}
\end{prop}

\noindent One can also show that $\dim \widetilde{\ba} = \displaystyle
\frac{(n+1)^5 - (n+1)^3}{24}$ (although this is not used in the paper).

\begin{proof}\hfill
\begin{enumerate}
\item Using Proposition \ref{Pproj}, Remark \ref{Rgen}, and the
previously developed theory of Category $\calo$, note that if $\varphi :
P_j / P_k \hookrightarrow P_r / P_k$, with the image of $\varphi$ denoted
by $V \subset P_r / P_k$, then $u^{\lambda_j - \lambda_r} 1_{P_r / P_k} =
\varphi(1_{P_j / P_k})$, where $1_{P_r / P_k}$ is the image of
$\overline{1} \in P(\lambda_r, \lambda_n - \lambda_r + 1) \in \calo$.
Thus we may set $r=j$ without loss of generality. We can also assume that
$s \leq n$.

Now let $M \subset P_j / P_k$ denote the submodule generated by
$d^{\lambda_j - \lambda_s} 1_{P_j / P_k}$. Clearly $M \twoheadrightarrow
M_s$ under the surjection $P_j / P_k \twoheadrightarrow P_j / P_{j+1}
\cong M_{j+1}$. Thus $P_s$ maps onto the cyclic $A$-module $M$ by
projectivity, yielding a morphism $\varphi : P_s \twoheadrightarrow M
\subset P_j / P_k$ whose image does not lie in $P_{j+1} / P_k$ (since the
image is $M_s \neq 0$). By the analysis in the proof of Theorem
\ref{Tproj}(1), $\varphi$ factors through an injective map $: P_s /
P_{s+k-j} \hookrightarrow P_j / P_k$, whose image equals $M$.

\item First suppose $\varphi : P_r / P_s \to P_j / P_k$ is a nonzero
morphism. Then $[P_j / P_k : L_r] > 0$ by Equation \eqref{Emult}, which
shows using Proposition \ref{Pproj} that $r < k$. Now suppose for the
remainder of this part that $r < k$.
Clearly, $\varphi_{(r,s),(j,k)}^{(t)}$ is an $A$-module morphism whose
image is contained in $P_{k-t} / P_k$. Moreover, the image is not
contained in $P_{k-t-1}/P_k$, by using the analysis after Equation
\eqref{Edual-verma}. Thus, the maps $\varphi_{(r,s),(j,k)}^{(t)}$ are
linearly independent, which shows that
\begin{equation}
\dim \Hom_\calo(P_r / P_s, P_j / P_k) \geq \mathbf{1}(r<k)
\min(s-r,k-r,k-j).
\end{equation}

We now show that the above maps also span the $\Hom$-space. Indeed,
suppose $\varphi \in \Hom_\calo(P_r / P_s, P_j / P_k)$; then
composing with the surjection $: P_r \twoheadrightarrow P_r / P_s$ yields
a map in $\Hom_\calo(P_r, P_j / P_k)$. By Equation \eqref{Emult}, this
latter space has dimension $[P_j / P_k : L_r] = \min(k-r,k-j)$. Thus,
assume for each $t \in (s-r, \min(k-r,k-j)]$ that $\varphi_t : P_r /
P_{r+t} \to P_j / P_k$ is a morphism with image in $P_{k-t} / P_k$.
Repeatedly applying Proposition \ref{Pses} shows that the nonzero
submodule $P_s / P_{r+t}$ embeds into the submodule $P_{k-t+s-r} / P_k
\neq 0$, but not in $P_{k-t+s-r+1}/P_k$, once again using the analysis
after Equation \eqref{Edual-verma} as well as Remark \ref{Rgen}. It
follows that no linear combination of the $\varphi_t$ is a map between
$P_r / P_s$ and $P_j / P_k$. Thus the maps $\varphi_{(r,s), (j,k)}^{(t)}$
(with $1 \leq t \leq \min(s-r, k-r, k-j)$) form an $\F$-basis of
$\Hom_\calo(P_r / P_s, P_j / P_k)$ for all $(r,s), (j,k)$. Consequently,
the maps $f_{jk}^{++}, f_{jk}^{- \bullet}, f_{jk}^{\bullet -}$ generate
$\widetilde{\ba}$, by Equation \eqref{Ephi}.

Now consider $\varphi_{(j,k), (a,b)}^{(u)} \circ \varphi_{(r,s),
(j,k)}^{(t)}$ for $1 \leq a < b \leq n+1$ and suitable $u>0$. The image
under $\varphi_{(r,s), (j,k)}^{(t)}$ of the generator $1_{P_r / P_s}$
lies in
\[
P_r / P_{r+t} \hookrightarrow P_{k-t} / P_k \hookrightarrow P_j / P_k,
\]

\noindent so we now ask where this generator goes under the surjection $:
P_j / P_k \twoheadrightarrow P_j / P_{j+u}$ (which is the first factor of
the composite map $\varphi_{(j,k), (a,b)}^{(u)}$). By the previous part,
the generator of $1_{P_r / P_{r+t}}$ in $P_j / P_k$ is precisely
\[
d^{\lambda_{k-t} - \lambda_r} u^{\lambda_{k-t} - \lambda_j} 1_{P_j /
P_k},
\]

\noindent so under the surjection $: P_j / P_k \twoheadrightarrow P_j /
P_{j+u}$, this generator goes to
\[
d^{\lambda_{k-t} - \lambda_r} u^{\lambda_{k-t} - \lambda_j} 1_{P_j /
P_{j+u}}.
\]

\noindent Once again applying the previous part, this vector generates
the submodule $P_r / P_{r+v} \hookrightarrow P_a / P_b$, with
\[
v := u+t+j-k.
\]

\noindent Thus, the composite map is nonzero if and only if $v > 0$, in
which case it sends $P_r / P_s \twoheadrightarrow P_r / P_{r+v}
\hookrightarrow P_{b-v} / P_b \hookrightarrow P_a / P_b$. We now verify
that $v$ is indeed at most $\min(s-r, b-r, b-a)$:
\begin{align*}
t \leq & \min(s-r, k-r, k-j), \quad u \leq \min(b-a, b-j, k-j)\\
\implies \quad v = &\ u + (t+j-k) \leq u \leq b-a;\\
v = &\ t + (u+j-k) \leq t \leq s-r;\\
v = &\ t+u + j-k \leq k-r + b-j + j-k = b-r.
\end{align*}

\noindent Thus we have shown that Equation \eqref{Egrading} holds.
The proof concludes by observing that Equations \eqref{Ephi},
\eqref{Egrading2}, and \eqref{Egrading} show that 
$\End_\calo(\widetilde{\bp})$ is indeed $\Z_+$-graded.

\item Note that $\End_\calo(P_r / P_s)$ is a $\Z_+$-graded subalgebra of
$\End_\calo(\widetilde{\bp})$, from the previous part. We now claim that
the only idempotent is $\varphi_{(r,s), (r,s)}^{(s-r)} = \id_{P_r /
P_s}$, which would show the result. Indeed, observe by the previous part
that $\varphi_{(r,s), (r,s)}^{(s-r)}$ is the only endomorphism having
(graded) degree zero. Now it is standard to verify that if $0 \neq
\sum_{t \geq 0} c_t \varphi_t$ is an idempotent with $\deg \varphi_t =
t$, then $c_0 = 1$ and $c_t = 0$ for $t>0$.\qedhere
\end{enumerate}
\end{proof}

\subsection{Tilting objects and their submodules}

The results in the previous subsection enable the analysis of the modules
$P_r / P_s$ and the classification of their submodules, as well as of all
tilting modules in the block $\calo[\lambda]$. The classifications
require repeated use of the following result.

\begin{lemma}\label{Lverma}
Fix $1 \leq r < k \leq n+1$, and a vector $0 \neq x \in P_r / P_k$.
Let $j \in [r,k)$ and $s \in [1,j]$ denote the unique integers such that
(a) $x \in (P_j / P_k) \setminus (P_{j+1} / P_k)$, and
(b) $x \mod (P_{j+1}/P_k) \in P_j / P_{j+1} \cong M_j$ lies in $M_s
\setminus M_{s-1}$.
Then the submodule generated by $x$ in $P_r / P_k$ contains the unique
copy of the submodule $P_s / P_{s+k-j}$ of $P_j / P_k \hookrightarrow P_r
/ P_k$.
\end{lemma}

\noindent Note that the lemma extends the analysis in the proof of
Theorem \ref{Thom} (see Equation \eqref{Edual-verma} and thereafter).

\begin{proof}
We first claim that the copy of $P_s / P_{s+k-j}$ inside $P_j / P_k$ is
unique. The claim follows via a careful analysis of the space
$\Hom_\calo(P_s / P_{s+k-j}, P_j / P_k)$ and its distinguished graded
basis, via Proposition \ref{Pprojmaps}.

Next, $(Ax)/ [Ax \cap (P_{j+1} / P_k)] \cong M_s$ by Proposition
\ref{Pverma}. It follows by Proposition \ref{Pprojmaps}(1) that
$v_{j,r,s} + v' \in Ax$, for some $v' \in P_{j+1}/P_k$, where
\[
v_{j,r,s} := d^{\lambda_j - \lambda_s} u^{\lambda_j - \lambda_r} 1_{P_r /
P_k}.
\]

\noindent Since all objects in the block $\calo[\lambda]$ are weight
modules, we may further assume that $v'$ is a weight vector of weight
$\lambda_s$ (as is $v_{j,r,s}$). By Proposition \ref{Pprojmaps}(1) again,
$v + 1_{P_s / P_{s+k-j}} \in Ax$ for some weight vector $v \in
P_{j+1}/P_k$ of weight $\lambda_s$. Now observe that if $v \in (P_l/P_k)
\setminus (P_{l+1}/P_k)$ for some $l < k$, then $v$ is a weight vector of
weight $\lambda_s$ under the quotient map $: P_l/P_k \twoheadrightarrow
M_l$. It follows that $u \cdot v \in P_{l+1}/P_k$. Applying $u$
repeatedly, one observes that $u^{\lambda_{k-j+s-1} - \lambda_s}$ kills
$v$, and sends $1_{P_s / P_{s+k-j}}$ via Proposition \ref{Pprojmaps}(1)
to the generating highest weight vector in the Verma submodule
$M_{k-j+s-1} \subset P_s / P_{s+k-j}$. It follows that $M_{k-j+s-1}
\subset Ax$.

By a similar argument, $u^{\lambda_{k-j+s-2} - \lambda_s}$ sends $v$ to a
weight vector in $M_{k-j+s-1}$ and $1_{P_s / P_{s+k-j}}$ to a generator
of $P_{s+k-j-2} / P_{s+k-j}$. By the previous paragraph, it follows that
$P_{s+k-j-2} / P_{s+k-j} \subset Ax$. Proceeding inductively along these
lines, we conclude that $P_s / P_{s+k-j} \subset Ax$.
\end{proof}

We now prove another of our main theorems in the present paper.

\begin{proof}[Proof of Theorem \ref{Tproj}]\hfill
\begin{enumerate}
\item First observe that if $s=k=n+1$, then using Proposition
\ref{Pprojmaps} and Equation \eqref{Emult}:
\begin{equation}
\dim \Hom_\calo(P_r,P_j) = [P_j : L_r] = \min(n+1-r,n+1-j) = n+1 +
\min(-r,-j) = n+1 - \max(r,j).
\end{equation}

Next, apply $\Hom_\calo(-, P_j / P_k)$ to the short exact sequence $0 \to
P_s \to P_r \to P_r / P_s \to 0$, and note that the Euler characteristic
of the corresponding long exact sequence is zero. Thus, we compute using
the above analysis:
\begin{align}\label{Eproj}
&\ \dim \Ext^1_\calo(P_r/P_s, P_j/P_k)\\
= &\ \dim \Hom_\calo(P_s, P_j/P_k) - \dim \Hom_\calo(P_r, P_j/P_k) + \dim
\Hom_\calo(P_r/P_s, P_j/P_k)\notag\\
= &\ [P_j / P_k : L_s] - [P_j / P_k : L_r] + 
\mathbf{1}(r<k) \min(s-r,k-r,k-j)\notag\\
= &\ \max(k,s) - \max(j,s) + \max(j,r) - \max(k,r) +
\mathbf{1}(r<k) \min(s-r,k-r,k-j).\notag
\end{align}

We now explicitly compute the last expression in \eqref{Eproj}, and
verify that it equals precisely $\mathbf{1}(r \leq j) \mathbf{1}(s \leq
k) (\min(0, j-s) + \min(s-r,k-j))$. There are three possible cases, and
in each of them the verification is straightforward:
\begin{enumerate}
\item $r \leq j, s \leq k$:
In this case, the last expression in \eqref{Eproj} equals
\[
k - \max(j,s) + j - k + \min(s-r,k-r,k-j) = \min(0,j-s) + \min(s-r,k-j).
\]

\item $j < r < k$:
In this case, the last expression in \eqref{Eproj} equals
\[
\max(k,s) - s + r - k + \min(s-r,k-r) = - \min(k,s) + r + \min(s,k) - r =
0.
\]

\item $r \not\in (j,k]; s > k$:
In this case, the last expression in \eqref{Eproj} equals
\[
s - s + \max(j,r) - \max(k,r) + \mathbf{1}(r \leq j) \min(s-r,k-r,k-j),
\]

\noindent and this is easily verified to equal $0$ in the two sub-cases:
$r \leq j$ and $r > k$.
\end{enumerate}

The above three cases prove the second of the Equations
\eqref{Eproj-dim}. Finally, to compute the higher $\Ext$-groups, a
similar computation to the ones above, using $\Hom_\calo(-,X)$ for any $X
\in \calo[\lambda]$, reveals that $\Ext^l_\calo(P_j/P_k,X) = 0$ for $X
\in \calo[\lambda]$ and $l \geq 2$.

\item The first observation is that every submodule $N \subset P_r / P_s$
has a filtration:
\[
0 \subset N \cap (P_{s-1}/P_s) \subset N \cap (P_{s-2} / P_s) \subset
\cdots \subset N \cap (P_r/P_s) = N,
\]

\noindent whose subquotients are submodules of Verma modules $(P_j/P_s) /
(P_{j+1}/P_s) \cong M_j$, hence are Verma modules themselves. It follows
that every submodule $N \subset P_r / P_s$ has a Verma flag.

Next, if $N \neq 0$, then $N \cap (P_{s-1}/P_s)$ is a nonzero submodule
by Lemma \ref{Lverma} and Proposition \ref{Pprojmaps}, so it necessarily
contains the submodule $L_1 \subset M_{s-1} = P_{s-1}/P_s \subset
P_r/P_s$. It follows that every submodule of $P_r/P_s$ is indecomposable.

Finally, we produce a bijection between the set of submodules $N \subset
P_r/P_s$ and the decreasing sequences specified in the statement of the
result. Given a decreasing sequence $s-1 \geq m_l > \cdots > m_1 \geq 1$,
construct the corresponding module $N \subset P_r/P_s$ as follows: first
set $N_s = 0$. Now given $N_j$ for some $s-l < j \leq s$, define
$N_{j-1}$ to be the lift to $N_j$ of the Verma submodule $M_{m_{j+l-s-1}}
\subset M_{j-1} \cong (P_{j-1}/P_s) / (P_j/P_s)$. In other words,
\[
0 \to N_j \to N_{j-1} \to M_{m_{j+l-s-1}} \to 0.
\]

\noindent In constructing $N_{j-1}$, we crucially use Lemma \ref{Lverma},
since the lift of $M_{m_{j+l-s-1}}$ to $P_r/P_s$ equals $P_{m_{j+l-s-1}}
/ P_{m_{j+l-s-1} + s-r}$, and since $\ker(P_{m_{j+l-s-1}} /
P_{m_{j+l-s-1} + s-r} \twoheadrightarrow M_{m_{j+l-s-1}})$ is necessarily
contained in $N_j$ by the hypothesis that the sequence of $m_j$ is
strictly decreasing. Proceeding inductively, we obtain the desired
submodule $N = N_{s-l} \subset P_{s-l}/P_s \subset P_r/P_s$.

Conversely, given $N \subset P_r/P_s$, let $l$ be the unique integer such
that $N \subset P_{s-l}/P_s$ but $N \subsetneq P_{s-l+1}/P_s$. Now
consider the filtration
\[
0 \subset N \cap (P_{s-1}/P_s) \subset N \cap (P_{s-2} / P_s) \subset
\cdots \subset N \cap (P_{s-l}/P_s) = N,
\]

\noindent whose subquotients are submodules of Verma modules $(P_j/P_s) /
(P_{j+1}/P_s) \cong M_j$, hence are Verma modules themselves. Denote the
subquotients as $M_{m_{s-1}}, \cdots, M_{m_{s-l}}$. Now choose $1 \leq j
\leq l$, and any weight vector $n'_j \in N \cap (M_{s-j}/M_s)$ whose
image modulo $N \cap (P_{s-j+1}/P_s)$ is the highest weight vector in
$M_{m_{s-j}}$. It follows by applying Lemma \ref{Lverma} to $x = n'_j$
that $s-1 \geq m_{s-1} > m_{s-2} > \cdots > m_{s-l} \geq 1$. Now it is
not hard to show that the two maps are inverse assignments, leading to
the aforementioned bijection.

In particular, the number of such submodules equals the number of such
decreasing subsequences, which can be of lengths $l = 0, 1, \dots, s-r$.
It follows that there are precisely $\displaystyle \sum_{l=0}^{s-r}
\binom{s-1}{l}$ such sequences.

\item \textbf{Step 1:}
We first show the following more general result:\medskip

\textit{Given integers $1 \leq j \leq r \leq s \leq n+1$, define
$M_{r,s} := P_1 / P_{1+s-r}$.
Then $M_{r,s}$ embeds into $P_j / P_s$, and its
cokernel has a finite filtration, with subquotients}
\begin{align*}
&\ F(M_{s-r}), F(M_{s-r+1}), \dots, F(M_{s-j-1}), F(M_{s-j});\\
&\ F(M_{s-j+1} / M_1), F(M_{s-j+2} / M_2), \dots, F(M_{s-1} / M_{j-1}).
\end{align*}

\noindent To show the above result, we begin by observing via Proposition
\ref{Pses} that
\[
M_{r,s} = P_1 / P_{1+s-r} \hookrightarrow P_2 / P_{2+s-r}
\hookrightarrow P_1 / P_{2+s-r},
\]

\noindent and the subquotients are $F(M_{s-r}/M_1)$ and $L_1 = F(M_1)$
respectively. Moreover, if we denote by $v_1$ the generator $1_{P_1 /
P_{2+s-r}}$, then for all $1 \leq t \leq 1+s-r$, the vector $u^{\lambda_t
- \lambda_1} v_1$ does not lie in $M_{r,s}$ by the analysis in
Proposition \ref{Pses}. It follows by lowest weight theory that the
cokernel of the inclusion $P_1 / P_{1+s-r} \hookrightarrow P_1 /
P_{2+s-r}$ is precisely $F(M_{s-r})$. Now change $s$ to $s+1, s+2, \dots$
in order to show that the inclusions
\[
M_{r,s} = P_1 / P_{1+s-r} \hookrightarrow P_1 / P_{2+s-r}
\hookrightarrow \cdots \hookrightarrow P_1 / P_{s-j+1}
\]

\noindent have respective cokernels equal to $F(M_{s-r}), F(M_{s-r+1}),
\dots, F(M_{s-j})$. This shows the first row of subquotients in the
statement above.
The second row of subquotients immediately follows by applying
Proposition \ref{Pses} to the successive inclusions
\[
P_1 / P_{s-j+1} \hookrightarrow P_2 / P_{s-j+2} \hookrightarrow \cdots
\hookrightarrow P_j / P_s.
\]

\noindent \textbf{Step 2:}
We now conclude the proof. Note from Proposition \ref{Pprojmaps} below
that $P_j / P_k$ is indecomposable for all $1 \leq j \leq k \leq n+1$.
Now apply the previous step with $j=1$ and $s=r=k+1$ for some $1 \leq k
\leq n$. It follows that $M_{k,k} = 0 \hookrightarrow P_1 / P_{k+1}$, and
the cokernel has a dual Verma flag. Thus $T_k := P_1 / P_{k+1}$ is an
indecomposable tilting module. Using Equation \eqref{Emult}, it is easily
verified that
\begin{equation}\label{Etilting}
[T_k : L_j] > 0 \Rightarrow j \leq k; \qquad
[T_k : L_k] = 1, \qquad \forall 1 \leq k \leq n.
\end{equation}

\noindent It follows by \cite[Theorem A4.2(i)]{Don} that every
indecomposable tilting module is isomorphic to $T_k$ for a unique $1 \leq
k \leq n$. Next, $F(T_k)$ is also an indecomposable tilting module, so
$F(T_k) \cong T_k$ by \eqref{Etilting}.
Finally, observe from Equation \eqref{Ephi} that there is a unique $t =
s-r$ such that $\varphi = \varphi_{(1,k), (1,n+1)}^{(t)} : T_{k-1}
\hookrightarrow T_n$ is injective. Moreover, every other map
$\varphi_{(1,k), (1,n+1)}^{(t)}$ has image contained inside
$\im(\varphi)$. 
Now dualize the short exact sequence $0 \to P_k \to T_n \to T_{k-1} \to
0$ to obtain
\[
0 \to T_{k-1} \to T_n \to F(P_k) \to 0,
\]

\noindent which concludes the proof by generalities in the highest weight
category $\calo[\lambda]$.
\end{enumerate}
\end{proof}

This classification of tilting modules sets the stage for employing the
comprehensive machinery developed by Ringel. We refer the reader to
\cite[Section A4]{Don} for some of the consequences for (tilting theory
in) the block $\calo[\lambda]$. Here we present one application of
Theorem \ref{Tproj}.

\begin{cor}\label{Ctilting}
The tilting modules form an increasing chain: $T_1 = L_1 \subset T_2
\subset \cdots \subset T_n = P_1$; dually, $T_n \twoheadrightarrow
T_{n-1} \twoheadrightarrow \cdots \twoheadrightarrow T_1$. Moreover, the
following are equivalent for a module $M$ in $\calo[\lambda]$:
\begin{enumerate}
\item $M$ is a submodule of a tilting module $T_k$ for some $1 \leq k
\leq n$.
\item $F(M)$ is a quotient of $T_k$.
\end{enumerate}
In this case, $F(T_k / M)$ is a submodule of $T_k$, hence indecomposable
with a Verma flag.
\end{cor}

\section{Koszulity and the SKL condition}\label{S5}

In this section we show that the endomorphism algebra $\ba$ is Koszul and
satisfies the Strong Kazhdan-Lusztig condition. The first step is to use
the analysis in the preceding sections to prove our remaining main
result.

\begin{proof}[Proof of Theorem \ref{Tkoszul}]
All assertion in the first part, except for the grading, Koszulity, and
dimension of $\ba$ follow from the analysis in Section \ref{SbggO}. Now
note that $\ba \subset \End_\calo(\widetilde{\bp})^{op}$, as discussed in
Proposition \ref{Pprojmaps}. Thus $\ba$ inherits the grading in
Proposition \ref{Pprojmaps}. In particular, the $\Z_+$-graded vector
space $V_{ij} := \Hom_\calo(P_i, P_j)$ has an $\F$-basis of maps
\[
\varphi_{(i,n+1), (j,n+1)}^{(u)}, \quad 1 \leq u \leq \min(n+1-i,
n+1-j) = n+1 - \max(i,j).
\]

\noindent The grading here is given by $\deg \varphi_{(i,n+1),
(j,n+1)}^{(u)} = 2(n+1-u) - i - j$. Now define the Hilbert matrix of
$\ba$ to be $H(\ba,t)_{i,j} := \sum_{l \geq 0} t^l \dim V_{ij}[l]$ (with
respect to this grading). It follows by the above analysis and Corollary
\ref{Cext} that
\[
H(\ba, t)_{ij}
= \sum_{u = \max(i,j)}^n t^{2u - i - j} = (H(E(\ba),-t)^{-1})_{ij}.
\]

\noindent Moreover, the algebra $\ba$ has zero-degree graded component
equal to the $\F$-span of
$\varphi_{(j,n+1), (j,n+1)}^{(n+1-j)}$ $= \id_{P_j}$, for all $1 \leq j
\leq n$. Thus, $\ba[0]$ is a semisimple algebra that contains a copy of
the unit in $\ba$: $1_{\ba} = \sum_{i=1}^n \id_{P_j}$. Using the
numerical criterion for Koszulity from \cite[Theorem 2.11.1]{BGS}, it
follows that $\ba$ is Koszul. This shows the first part of the theorem.
(It also follows by \cite[Section 2.5]{BGS} that $\ba$ is the associated
graded algebra of its radical filtration.)

For the second part, first define the quiver algebra with relations
$Q_\lambda$ to be the quotient of the path algebra of the double of the
$A_n$-quiver, modulo the relations \eqref{Erelations}. Now note that the
$\Ext$-quiver of $\ba$ is as claimed, by Corollary \ref{Cext}.
Next, define $\gamma_i, \delta_i \in \ba$, and the idempotent zero-length
path $e_i$ at $[i]$, as follows:
\[
e_i \leftrightarrow \varphi_{(i,n+1), (i,n+1)}^{(n+1-i)}, \qquad
\gamma_i \leftrightarrow \varphi_{(i+1,n+1), (i,n+1)}^{(n-i)} =
f_{i+1,n+1}^{- \bullet}, \qquad
\delta_i \leftrightarrow \varphi_{(i,n+1), (i+1,n+1)}^{(n-i)} =
f_{i,n+1}^{++}.
\]

\noindent Using the explicit relations \eqref{Egrading} satisfied by the
maps $\varphi_{(r,s), (j,k)}^{(t)}$, it follows that $Q_\lambda
\twoheadrightarrow \ba^{op}$ as $\Z_+$-graded algebras. On the other
hand, it is easily verified that $\dim Q_\lambda = 1^2 + \cdots + n^2 =
\dim \ba^{op}$. This concludes the proof.
\end{proof}

\begin{remark}\label{Rtgwa}
The assumption that $z_1 \in H^\times$ was required in order to equip
$\cals(H,\theta,z_0,z_1)$ with a GWA structure. Thus, algebras defined by
\eqref{Egwareln} with $z_1 \in H \setminus H^\times$ are algebras with
triangular decomposition that are not GWAs, by Lemma \ref{Lgwa}.

We now remark that our main results, Theorems \ref{Tkoszul}--\ref{Tproj},
in fact hold for this more general family of algebras (given that $H$ is
commutative, whence $ud,du$ are transcendental over $H$). The proof in
this general setting involves certain explicit computations, which do not
require that $z_1 \in H^\times$; instead, it suffices to assume the
weaker condition that $\mu(z_1) \neq 0\ \forall \lambda_1 \leq \mu \leq
\lambda_n$. Specifically, these explicit computations occur only
in Section \ref{Sprelim}, and in proving Proposition \ref{Pprojmaps} and
Theorem \ref{Thom}; the remaining proofs go through unchanged.

Note that in several prominent examples in the literature listed above,
$\hfree = \widehat{H}$ (see \cite[Section 8]{Kh2} for more details). Thus
for our results to hold in all blocks of $\calo$ in such examples, we
would need to assume that $z_1$ does not belong to any maximal ideal
$\ker(\lambda)$ for $\lambda \in \widehat{H}$. In other words, $z_1 \in
H$ would need to be a unit, which explains the assumption in the present
paper.
\end{remark}

\begin{remark}\label{Rkoszul}
As discussed after the statement of Theorem \ref{Tkoszul}, the algebra
$\ba$ only depends on $n = |[\lambda]|$. Thus we define $\A_n := \ba$. 
For completeness, we briefly discuss other settings in the
literature where the family of algebras $\A_n$ is studied.
Note that $\A_n$ is the endomorphism algebra of the projective generator
of various (singular) blocks of Category $\calo$ over complex simple Lie
algebras of low rank; see \cite[Sections 6 and 7]{FNP} and \cite[Section
5]{St} for more details.
The algebra also features in the study of Category $\calo$ over the
Virasoro algebra, in finite blocks, or finite quotients or truncations of
thin blocks; this is discussed at length in \cite{BNW}. Furthermore, the
algebra $\A_n$ and its quadratic dual play a role in the study of
hyperplane arrangements, the hypertoric category $\calo$, perverse
sheaves on $\lie{gl}_n(\C)$, and Cherednik algebras. We refer the reader
to \cite{BEG,BG,BLPW} for more details.
\end{remark}

We next show that the algebra $\ba$ possesses an additional useful
homological property called the \textit{Strong Kazhdan-Lusztig
condition}. We begin with an Ext-computation that holds in greater
generality, in any highest weight category. Say that an object $X$ in the
block $\calo[\lambda]$ is in $\calf(\Delta)$ (respectively
$\calf(\nabla)$) if $X$ has a (dual) Verma flag.

\begin{theorem}[{\cite[Proposition A2.2]{Don}}]\label{Tdon}
For all $X \in \calf(\Delta), Y \in \calf(\nabla)$, we have:
\[
\dim \Ext^l_\calo(X,Y) = \delta_{l,0} \sum_{i=1}^n [X : M_i][Y : F(M_i)].
\]

\noindent Moreover, an object $X \in \calo[\lambda]$ is in
$\calf(\Delta)$, if and only if $\Ext^1_\calo(X,F(M_i)) = 0\ \forall 1
\leq i \leq n$.
\end{theorem}

We now discuss the Strong Kazhdan-Lusztig condition in the block
$\calo[\lambda]$, where $[\lambda]$ is finite.

\begin{defn}[{\cite[\S 2.1]{CPS2}}]
Given a finite length $A$-module $M$, define its \textit{radical} and
\textit{socle} filtrations by:
\begin{align*}
&\ \Rad^0 M = M, \qquad \Rad^j M = \Rad(\Rad^{j-1} M),\\
& \ \Soc^0 M = M, \qquad \Soc^j M / \Soc^{j-1} M = \Soc(M / \Soc^{j-1}
M), \qquad j > 0.
\end{align*}

\noindent The block $\calo[\lambda]$ satisfies the \textit{Strong
Kazhdan-Lusztig condition (SKL)} relative to a given function $\ell :
[\lambda] \to \Z$, if for all integers $0 \leq l,i$ and $1 \leq j,k \leq
n$,
\begin{align*}
&\ \Ext^l_\calo(M_j, \Soc^i(F(M_k))) \neq 0 \mbox{ or }
\Ext^l_\calo(\Rad^i(M_j), F(M_k)) \neq 0\\
\implies &\ l \equiv \ell(\lambda_j) - \ell(\lambda_k) + i \mod 2.
\end{align*}
\end{defn}

Kazhdan-Lusztig theories and conditions such as (SKL) are desirable
properties to have in a highest weight category (equivalently, for
quasi-hereditary algebras). A large program has been developed in the
literature by Cline, Parshall, and Scott whereby they discuss how such
conditions can be attained, as well as specific consequences of having
such a theory at hand. See \cite{CPS1,CPS2} and the references therein
for more information on the subject.

We conclude this section by proving the Strong Kazhdan-Lusztig condition
for the block $\calo[\lambda]$.

\begin{prop}
If Assumption \ref{Stand3} holds, then $\calo[\lambda]$ satisfies the
Strong Kazhdan-Lusztig condition with respect to the natural length
function $\ell : [\lambda] \to \Z$ given by $\ell(\lambda_j) := j =
l(M_j)$ for $1 \leq j \leq n$.
\end{prop}

\noindent In particular, $\ba^!$ is also Koszul (by results in
\cite{CPS2}), and $(\ba^!)^! \cong \ba$. Using Theorem \ref{Tbgs}, this
provides a second (albeit less direct) proof of the Koszulity of the
endomorphism algebra $\ba$.

\begin{proof}
By Proposition \ref{Pverma}, $\Rad^j M_i = M_{i-j}$ for all $j$. Now
compute using Theorem \ref{Tdon}:
\[
\dim \Ext_\calo^l(\Rad^i(M_j),F(M_k)) = \dim \Ext_\calo^l(M_{j-i},
F(M_k)) = \delta_{l,0} \delta_{j-i,k}.
\]

\noindent Thus, if $\Ext_\calo^l(\Rad^i(M_j), F(M_k))$ is nonzero, then
$l=0$ and $j-i=k$, whence $\ell(\lambda_j) - \ell(\lambda_k) + i - l =
j-k+i-0 = 2i \equiv 0 \mod 2$, as desired. Now note via the duality
functor $F$ that the socle series of $F(M_k)$ is dual to the radical
series of $M_k$, and hence is also uniserial. Thus, the condition
involving the socle filtration is verified as above, since
$\Soc^i(F(M_k)) = F(\Rad^i(M_k)) = F(M_{k-i})$.
\end{proof}

\section{The category of sub-triangular Young tableaux}\label{Sstyt}

We now introduce the notion of sub-triangular Young tableaux. This is a
hitherto unexplored phenomenon for triangular GWAs, which affords a
combinatorial interpretation of morphisms and extensions between
distinguished objects of the block $\calo[\lambda]$.

\subsection{The transfer maps}

We begin with the ``transfer map'' obtained from Theorem \ref{Tproj}(2),
which sends a submodule $N \subset P_r/P_s$ to an integer tuple $(s-1
\geq m_1 > \cdots > m_l \geq 1$), for some $0 \leq l \leq s-r$. Since
$P_r/P_s \hookrightarrow P_1/P_s \hookrightarrow P_1 = T_n$ via
Proposition \ref{Pprojmaps}(1), we can now define a map $\Psi$ from a
submodule of $P_1 = T_n$ to tuple of integers, via:
\begin{equation}
N \subset T_n  \quad \leadsto \quad \Psi(N) := (m_1, \dots, m_l).
\end{equation}

\noindent Moreover, the integers $m_j$ are obtained as follows: consider
the filtration
\[
0 \subset N \cap (P_{s-1}/P_s) \subset N \cap (P_{s-2} / P_s) \subset
\cdots \subset N \cap (P_r/P_s) = N.
\]

\noindent Each subquotient is a submodule of the Verma module $M_j \cong
(P_j/P_s) / (P_{j+1}/P_s)$, hence is a Verma module. Denote by $l$ the
number of nonzero subquotients, and by $M_{m_j}$ the subquotient of
$M_{s-j}$ for $1 \leq j \leq l$.
For instance, $\Psi(P_r/P_s) = (s-1, s-2, \dots, r)$ for all $1 \leq r <
s \leq n+1$, which includes all tilting, projective, and Verma modules in
the block $\calo[\lambda]$.

We now explain how to encode the transfer map $\Psi$ by Young tableaux.
First,
observe by the above discussion that the submodules $N \subset
T_n$ are in bijection with
Young tableaux with consecutively decreasing integer entries in
each column (to $1$)
and each row, and where the topmost cells of each row form the
sequence $\Psi(N)$. For
instance, the following figure corresponds to the submodule
$N_0 := \Psi^{-1}((5,3,2))$.
\[ \young(5,432,321,21,1) \]

This module is contained in $P_3/P_6$, and hence in any
$P_r/P_s$ into which $P_3/P_6$ embeds.
Moreover, as explained in the construction of the map $\Psi$,
the columns of the diagram
correspond to a Verma flag of $N_0$, and for each $j=1,2,3$,
the first $j$ leftmost columns
contain the Jordan-Holder factors in the corresponding
submodule of $N_0$.

Given any submodule $N \subset P_r/P_s$, define $\Y(N)$ to be the Young
tableau with strictly decreasing entries, which is obtained from
$\Psi(N)$ in the above manner.
We now observe that the map $\Y(\cdot)$ behaves well under taking the
quotient of one submodule of $P_r/P_s$ by another. Namely, it is not hard
to show that if $N' \subset N \subset P_r/P_s$, $N/N'$ has a filtration
whose subquotients are highest weight modules of the form $M_{m_j} /
M_{m'_j}$, where $\Psi(N) = (m_1, \dots, m_l)$ and $\Psi(N') = (m'_1,
\dots, m'_l)$ (with possibly some zeros added to the end to obtain
exactly $l$ entries). Thus it is natural to define $\Y(\cdot)$ at any
subquotient of $P_r/P_s$ (and hence of $T_n$), as the skew-tableau
\begin{equation}\label{Esubquot}
\Y(N/N') := \Y(N) \setminus \Y(N'),
\end{equation}

\noindent where $\Y(N')$ embeds into the first few leftmost columns of
$\Y(N)$, with each cell of $N'$ mapping to a cell in $N$ with the same
number. Note that such subquotients cover all objects in the block
$\calo[\lambda]$ that are generated by a single weight vector.

\subsection{Dual objects and dual Young tableaux}\hfill

We now show that the diagrams of dual objects are closely related --
in fact, they are transposes of one another. To examine this in
closer
detail, first recall that tilting objects are self-dual. This
is also reflected
in their corresponding Young tableaux, which we now give a
name.

\begin{defn}
Given an integer $k \geq 1$, define $\Y_k$ to be the
labelled triangular diagram:
\[
{\Large\young(\kmz\kmo\kmt\cdots21,\kmo\kmt\cdots\cdots1,\kmt\vdots\ddots\ddots,\vdots\vdots\ddots,21,1)}
\]
\end{defn}

Observe that $\Y_k = \Y(T_k)$ for all $1 \leq k \leq n$. As is standard,
we will denote the conjugate, or transpose, of a Young tableau $X$ by
$X^T$. Then $\Y_k = \Y(T_k)$ is self-conjugate for each $k$,
corresponding to the self-duality of $T_k$. We now show the relation
between the Young tableau corresponding to a subquotient of $T_k$ and its
dual, by refining Corollary \ref{Ctilting}.

\begin{prop}\label{Pstyt}
Suppose $N \subset T_k = P_1 / P_{k+1}$ is a submodule for some $1 \leq k
\leq n$, and $\psi := \Psi(N)$. Then we have the following short exact
sequence:
\begin{equation}\label{Etiltsub}
0 \to N = \Psi^{-1}(\psi) \to T_k = \Psi^{-1}((k, \dots, 1)) \to T_k/N =
F(\Psi^{-1}(\{ k, \dots, 1 \} \setminus \psi)) \to 0,
\end{equation}

\noindent i.e., $\Psi(F(T_k/N))$ equals the set $\{ 1, \dots, k \}
\setminus \Psi(N)$ arranged in decreasing order. In particular, if $N'
\subset N \subset P_r/P_s \subset T_n$ are $A$-submodules, then $N/N',
F(N/N')$ are subquotients of $T_n$, and
\begin{equation}
\Y(F(N/N')) = Y(N/N')^T = \Y(N)^T \setminus \Y(N')^T.
\end{equation}

\noindent Moreover, the number of cells labelled $\framebox{j}$ in
$\Y(N/N')$ is precisely the multiplicity $[N/N' : L_j]$ for all $1 \leq j
\leq n$. In particular, the total number of cells in $\Y(N/N')$ is
precisely the length of $N/N'$.
\end{prop}

\noindent The underlying combinatorial phenomenon is as follows, and
verifiable by direct visual inspection: for any submodule $N \subset
T_k$, the diagram $\Y_k \setminus \Y(N)$ is the transpose of $\Y(N_1)$
for a submodule $N_1 \subset T_k$. The result says that in fact $N_1 =
F(T_k / N)$.

\begin{proof}
We begin by refining the proof of Theorem \ref{Tproj}(2). In what
follows, we use Proposition \ref{Pprojmaps}(1) and Lemma \ref{Lverma}
without further reference. Suppose $\Psi(N) = (m_1, \dots, m_l)$. We
claim that $N \hookrightarrow P_1 / P_{m_1+1} \hookrightarrow P_1 /
P_{k+1}$. Indeed, let $X_j$ denote the image of $P_1 / P_{j+1}
\hookrightarrow P_{k-j+1}/P_{k+1} \hookrightarrow P_1 / P_{k+1} = T_k$
under the map $\varphi_{(1,j+1), (1,k+1)}^{(k-j)}$,
and let $x_j$ denote the image of $1_{P_1 / P_{j+1}}$ in the isomorphic
copy $X_j \subset T_k$. Now observe that $u^{\lambda_{m_1} - \lambda_1}
x_{m_1}$ generates the Verma module $M_{m_1} \subset N \cap X_{m_1}$.
Similarly, $d^{\lambda_{m_1 - 1} - \lambda_{m_2}} u^{\lambda_{m_1 - 1} -
\lambda_1} x_{m_1}$ generates the Verma submodule $M_{m_2} \subset
M_{m_1-1}$ of the module $(N / M_{m_1}) \cap (X_{m_1} / M_{m_1})$.
Proceeding inductively, the claim that $N \subset X_{m_1}$ follows.

Now we begin the proof. Note by Corollary \ref{Ctilting} that $F(T_k/N)$
is a submodule of $F(T_k) = T_k$. We first show by induction on
$|\Psi(N)|$ that $\Psi(F(T_k/N)) = \{ 1, \dots, k \} \setminus \Psi(N)$.
The proof is by induction on $|\Psi(N)|$. Thus, given $N \subset T_k$,
note from above that $N \subset X_{m_1} \subset T_k$, so by Step (1) of
the proof of Theorem \ref{Tproj}(3), $T_k / N$ has a finite filtration
with subquotients $F(M_k), \dots, F(M_{m_1+1})$, and $X_{m_1} / N$.
Because $M_{m_1} \hookrightarrow N \hookrightarrow X_{m_1} \cong P_1 /
P_{m_1+1}$, it follows that $X_{m_1} / N \cong (P_1 / P_{m_1}) / (N /
M_{m_1})$. Since $\Psi(N/M_{m_1})$ has strictly smaller length than
$\Psi(N)$ (including in the case when $|\Psi(N)| = 1$), and since we are
left to deal with $N / M_{m_1} \hookrightarrow P_1 / P_{m_1} = T_{m_1 -
1}$, we are done by induction.

In particular, the above analysis applied to $F(T_k/N)$ shows Equation
\eqref{Etiltsub}, and also that $\Y(F(N)) = \Y(N)^T$ for all sub-objects
$N \subset T_n$. Next, given sub-objects $N' \subset N \subset T_n$, let
$C', C$ denote the cokernels of the maps $N' \hookrightarrow T_n$ and $N
\hookrightarrow T_n$ respectively. Hence by the above analysis and
Corollary \ref{Ctilting}, $\Y(F(C)) = \Y(C)^T$, and similarly for
$\Y(F(C'))$. Now $N/N' \hookrightarrow C' \twoheadrightarrow C$ is exact,
so $F(C) \hookrightarrow F(C') \twoheadrightarrow F(N/N')$ by duality,
and $F(C') \subset F(T_n)$. Therefore,
\begin{align*}
\Y(F(N/N')) = &\ \Y(F(C')) \setminus \Y(F(C)) = \Y(C')^T \setminus
\Y(C)^T\\
= &\ (\Y_n \setminus \Y(N'))^T \setminus (\Y_n \setminus \Y(N))^T\\
= &\ \Y(N)^T \setminus \Y(N')^T = \Y(N / N')^T.
\end{align*}

Finally, it suffices to prove the assertion about multiplicities for
submodules $N \subset P_1$. But this follows from the detailed analysis
of the Verma flag of $N$ as described in Theorem \ref{Tproj}(2) (and
earlier in this section). \qedhere
\end{proof}

Given this compatibility between dual objects and their associated (dual)
Young tableaux, it is natural to ask if these connections can be made
precise. In the rest of this section, our goal is to provide a positive
answer by introducing a category with such diagrams as objects, and
suitable candidates for morphisms. As we will see, we achieve more, by
providing combinatorial analogues of the distinguished morphisms
$\varphi_{(r,s),(j,k)}^{(t)}$ (see Equation \eqref{Ephi} and Proposition
\ref{Pprojmaps}).

\subsection{Objects in the category of sub-triangular Young tableaux}

We now introduce and study a combinatorial category $\ycat$ of
sub-triangular Young tableaux, that will contain the aforementioned
diagrams corresponding to subquotients of $T_n$. In this subsection we
analyze the objects of $\ycat$, and show that they include the diagrams
$\Y(N/N')$ discussed above.

\begin{defn}\hfill
\begin{enumerate}
\item Define a \textit{sub-triangular Young tableau (STYT)} to be a
diagram $X$ that satisfies the following properties:
\begin{enumerate}
\item $\framebox{\it{k}} \subset X \subset \Y_k$ for some $k \geq 1$.

\item $X$ is connected.

\item For every row $R$ and column $C$ of $\Y_k$, the sub-diagrams $X
\cap R$ and $X \cap C$ are connected.

\item If $c$ is a cell in $\Y_k \setminus X$, then $X$ cannot contain the
cells immediately above $c$ and to the immediate left of $c$, if both
cells exist in $\Y_k$.
\end{enumerate}

\item Given a cell $c \in X$, denote the number in it by $n(c)$.
\end{enumerate}
\end{defn}

\noindent Here is an example of a STYT:
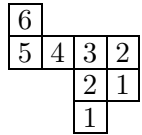
\[ \young(6,5432,::21,::1) \]
\noindent By Proposition \ref{Pstyt}, this diagram equals
$\Y(\Psi^{-1}((6,4,3,2)) / \Psi^{-1}((4,3)))$.

Henceforth, fix a triangular GWA satisfying Assumption \ref{Stand3}, and
a block $\calo[\lambda]$ with $[\lambda] = \{ \lambda_1 < \cdots <
\lambda_n$ for some $n \geq 1$.
Given integers $0 \leq j \leq k \leq n+1$, Proposition \ref{Pstyt} implies
that $\Y(M_k/M_j)$ and $\Y(F(M_k/M_j))$ (with $k \neq n$), and $\Y(P_j /
P_k)$ and $\Y(F(P_j / P_k))$ (with $j \geq 1$) are, respectively, the
following diagrams, which can all be verified to be STYTs:

\vspace*{0.4cm} \hspace*{0.1cm} {\Large\young(\kmz,\kmo,\svd,\jpo)}

\vspace*{-2.4cm} \hspace*{2.1cm}
{\Large\young(\kmz\kmo\scd\jpo)}

\vspace*{-0.6cm} \hspace*{5.9cm}
{\Large\young(\kmo\kmt\scd\jpz,\kmt\scd\scd\jmo,\svd\svd\sdd\svd,4321,321,21,1)}

\vspace*{-4.1cm} \hspace*{9.8cm}
{\Large\young(\kmo\kmt\scd4321,\kmt\scd\scd321,\svd\svd\sdd21,\jpz\jmo\scd1)}

\vspace*{1.8cm}

\begin{remark}
By the analysis in Propositions \ref{Pses} and \ref{Pstyt}, as well as
the proof of Theorem \ref{Tproj}(3), a group of rows at the bottom or a
group of columns on the left denotes STYT diagrams of sub-objects in
$\calo[\lambda]$, while a group of rows at the top or a group of columns
on the right denotes STYT diagrams of quotients in the block. In order to
maintain descending numbers as one moves right or down, the objects in
$\calf(\Delta)$ (respectively, in $\calf(\nabla)$) are written so that
their (dual) Verma subquotients in a (co)standard filtration occur as
columns (respectively, rows) of the corresponding STYTs.
\end{remark}

\begin{remark}
Observe that if we relabelled the families of objects $\{ L_j, M_j, P_j,
T_j \}$ under the permutation $w_\circ = (j \leftrightarrow n+1-j)$ of
$\{ 1, \dots, n \}$, then the STYT diagrams would consist of standard
Young tableaux, with strictly increasing (and successive) integers in
each row and column. In a parallel representation-theoretic setting
involving quantum groups, Young tableaux have connections to crystals;
see e.g.~\cite{HL} and the references therein. In our setting, the cells
in a sub-triangular Young tableau correspond not to an $\F$-basis for a
representation, but to the set of Jordan-Holder factors of the
corresponding representation, using Proposition \ref{Pprojmaps}(1) via
the transfer map $\Psi$.
\end{remark}

We now show that the notion of STYTs is the same as that of $\Y(\cdot)$.

\begin{prop}
For all STYTs $X \subset \Y_k$, there exist submodules $N' \subset N
\subset T_k$ such that $X = \Y(N/N')$. However, the converse is not
necessarily true.
\end{prop}

\noindent Note also that to each pair of modules $N' \subset N \subset
T_k$, there corresponds an integer $1 \leq l \leq k$, and two sequences
of decreasing integers $k \geq b_l > b_{l-1} > \cdots > b_1 \geq 1$ and
$b_l > a_l \geq a_{l-1} \geq \cdots \geq a_1 \geq 0$, such that $a_j >
a_{j-1}$ whenever $a_{j-1} > 0$, and $b_j > a_j$ for all $j$. The $b_j$
form $\Psi(N)$ and the nonzero $a_j$ form $\Psi(N')$, respectively.

\begin{proof}
Given an STYT $X \subset \Y_k$, it follows easily from the definition of
an STYT that the top entries in each column are strictly decreasing,
starting at $k$. Define $N$ to be the corresponding submodule of $T_k$.
Let the last entry in the $j$th column be denoted by $a_j$; then if
$a_{j-1} > 0$, it follows by the definition of an STYT that $a_j >
a_{j-1}$. Thus, if $a_1 > a_2 > \cdots > a_r$ denote the entries
among the $a_j$ that are not $1$, then Equation \eqref{Esubquot} implies
that $X = \Y(N/N')$, where $N' = \Psi^{-1}((a_1-1, \dots, a_r-1))$. This
proves the first assertion; the converse is, however, not true, as is
verified from following easy example: $N = \Psi^{-1}((3,1)), N' =
\Psi^{-1}((2))$.
\end{proof}

Having shown that the assignment $\Y(\cdot)$ is compatible with taking
subquotients and duals, we now discuss additional properties of
$\Y(\cdot)$ related to generators. We require the following notation.

\begin{defn}\hfill
\begin{enumerate}
\item Given a subset $X'$ of cells in a STYT $X$, the \textit{STYT
generated by $X'$ in $X$}, denoted by $\Y(X',X)$, is the sub-diagram
consisting of all cells obtained by traveling from a cell in $X'$ via a
finite sequence of moves, either one cell to the left, or one cell down.

\item Given a STYT $X$, define its set of \textit{primitive generators}
to be any minimal subset of cells $G_{\min}(X) \subset X$ such that $X =
\Y(G_{\min}(X),X)$.
\end{enumerate}
\end{defn}

\begin{lemma}
The primitive generating set of any STYT $X$ is unique. If $X =  \Y(N)$
for $N$ or $F(N)$ of the form $M_k / M_j$ or $P_j / P_k$, then
$G_{\min}(X)$ is a single cell.
\end{lemma}

\begin{proof}
That every STYT has a minimal generating set is obvious since $X$ has
only finitely many cells; that this set is unique follows by assigning a
coordinate to each cell that strictly increases upon moving one cell down
or to the left. Now it is clear that each of $\Y(M_k / M_j)$ and $\Y(P_j
/ P_k)$, or their duals, is generated by one cell, so we are done by the
uniqueness of $G_{\min}(X)$ for all $X$.
\end{proof}

Proposition \ref{Pprojmaps}(1) and Lemma \ref{Lverma} have combinatorial
interpretations using STYTs as well. Given any quotient of projective
modules $P_r / P_k$ and a nonzero element $x \in P_r / P_k$, first define
the \textit{cell of $x$}, denoted by $cell(x)$, as follows: consider $j
\in [r, k)$ such that $x \in (P_j/P_k) \setminus (P_{j+1}/P_k)$ inside
$P_r / P_k$. Now consider the submodule generated by $x$ inside the Verma
module $M_j \cong (P_r/P_k) / (P_{r+1}/P_k)$, say $M_s$ for $s \leq j$.
Then $cell(x)$ is defined to be the cell numbered $\framebox{{\it s}}$
that is $j-r$ steps to the left and $j-s$ steps below the generating cell
$G(\Y(P_r/P_k))$.

Now if $N$ is the submodule generated by $v_{j,r,s} := d^{\lambda_j -
\lambda_s} u^{\lambda_j - \lambda_r} 1_{P_r/P_k} \in P_r/P_k$, then one
has:
\[
\Y(N) = \Y(A \cdot v_{j,r,s}) = \Y(cell(v_{j,r,s})) \subset \Y(P_r /
P_k).
\]
\noindent The same result holds, albeit with a simpler proof, for the
highest weight module $M_r / M_s$ and for the lowest weight module $F(M_r
/ M_s)$.

\subsection{Morphisms and extensions of sub-triangular Young tableaux}

We now discuss morphisms and extensions between STYTs, as well as the
subcategory $\ycat$ of decreasing Young tableaux.

\begin{defn}\label{Dstyt}
Fix STYTs $X_1, X_2$.
\begin{enumerate}
\item Define a \textit{map} $\varphi : X_1 \to X_2$ to be a translation
(in the plane) of the diagram $X_1$, satisfying the following conditions:
\begin{enumerate}
\item $\varphi$ is $n$-equivariant: for all cells $c \in X_1$, either
$\varphi(c)$ is a cell in $X_2$ with $n(\varphi(c)) = n(c)$, or
$\varphi(c)$ is disjoint from $X_2$.
\item $\Y(\varphi(G_{\min}(X_1)), X_2) = X_2 \cap \varphi(\Y(X_1))$ is
non-empty.
\end{enumerate}

\item A map $\varphi : X_1 \to X_2$ is \textit{injective} if
$\varphi(X_1) \subset X_2$, and \textit{surjective} if $\varphi(X_1)
\supset X_2$.

\item A \textit{morphism} $: X_1 \to X_2$ is a formal $\F$-linear
combination of maps $: X_1 \to X_2$.

\item Define an \textit{extension} of $X_2$ by $X_1$ to be a disjoint
juxtaposition of $X_1$ and $X_2$ (but sharing at least one edge of one
cell), such that their (disjoint) union is a STYT, and $X_1$ is either
above or to the right of $X_2$.

\item Let $\Hom(X_1,X_2) = \Ext^0(X_1, X_2)$ denote the set of morphisms
from $X_1$ to $X_2$, and denote by $\Ext^1(X_1, X_2)$ the $\F$-span of
extensions of $X_2$ by $X_1$.

\item Define $\Y(\oplus_{j=1}^k N_j) := \coprod_{j=1}^k \Y(N_j)$ for all
objects $N_j \in \calo[\lambda]$ for which $\Y(N_j)$ is defined.
\end{enumerate}
\end{defn}

The following result provides connections between combinatorics and
blocks of triangular GWAs. To our knowledge, these connections have not
been explored in the literature.

\begin{theorem}\label{Tyst}
We work in the block $\calo[\lambda]$, where $[\lambda] = \{ \lambda_1 <
\cdots < \lambda_n \}$. Then the assignment $N \mapsto \Y(N)$ respects
morphisms and extensions, including under duality. More precisely, the
following analogues of Equations \eqref{Eext} and \eqref{Eprojext} hold:
\begin{equation}\label{Ecdsyt}
\dim \Ext^l_\calo(N,N') = \dim \Ext^l(\Y(N), \Y(N')), \quad \forall l=0,1,
\end{equation}

\noindent if $N,N'$ satisfy one of the following conditions:
\begin{enumerate}
\item $N = P_j / P_k$ as above, and $N' = M_r / M_s$ as above or $P_{j'}
/ P_{k'}$ for some $1 \leq j' < k' \leq n+1$;

\item $N,N'$ are both simple;

\item $N$ is projective and $N'$ is such that $\Y(N')$ is defined. If $N
= P_k$ for some $1 \leq k \leq n$ then $\dim \Hom(\Y(P_k), \Y(N'))$
equals the multiplicity of the cell $\framebox{k}$ in $\Y(N')$;

\item $N, F(N')$ are Verma modules;

\item or, if $(Y,Y')$ are one of the above pairs, and $N = F(Y'), N' =
F(Y)$. In other words,
\begin{equation}
\dim \Ext^l(\Y(N), \Y(N')) = \dim \Ext^l(\Y(N')^T, \Y(N)^T) = \dim
\Ext^l(\Y(F(N')), \Y(F(N))).
\end{equation}
\end{enumerate}
\end{theorem}

\noindent The result follows from the analysis carried out in this paper
in the various special cases (and by visual inspection of the
corresponding STYTs).
In fact the connection in  Theorem \ref{Tyst} is even stronger. Recall
via Proposition \ref{Pprojmaps} that the endomorphism algebra of
$\widetilde{\bp} := \bigoplus_{1 \leq r < s \leq n+1} P_r / P_s$ is
equipped with a $\Z_+$-grading, as well as a distinguished basis of
morphisms
\[
\varphi_{(r,s), (j,k)}^{(t)} : P_r / P_s \twoheadrightarrow P_r / P_{r-t}
\hookrightarrow P_{k-t} / P_k \hookrightarrow P_j / P_k.
\]

\noindent It is not hard to verify that the combinatorial counterparts of
these morphisms are precisely the distinct maps between the corresponding
STYTs:
\begin{equation}
\Y(\varphi_{(r,s), (j,k)}^{(t)}) : \Y(P_r / P_s)
\twoheadrightarrow \Y(P_r / P_{r+t}) \hookrightarrow \Y(P_{k-t} / P_k)
\hookrightarrow \Y(P_j / P_k).
\end{equation}

\noindent Moreover, the degree of the map $\varphi_{(r,s), (j,k)}^{(t)}$
precisely equals the Manhattan distance between the two (unique)
generating cells for the STYT map
$\displaystyle
\Y(\varphi_{(r,s), (j,k)}^{(t)}) : \Y(P_r / P_s) \to \Y(P_j / P_k)$,
i.e.,
\begin{align}
\deg \varphi_{(r,s), (j,k)}^{(t)} = &\ 2(k-t) - r - j\\
= &\ d_{\rm Manhattan} \left( G \left( \Y(\varphi_{(r,s),
(j,k)}^{(t)})(\Y(P_r / P_s)) \right), G(\Y(P_j / P_k)) \right). \notag
\end{align}

\noindent This also holds for the maps $\Y(f_{jk}^{++}), \Y(f_{jk}^{-
\bullet}), \Y(f_{jk}^{\bullet -})$ as in \eqref{Egrading2}. Thus, we have
shown:

\begin{prop}
The endomorphism algebra of $\Y(\widetilde{\bp}) = \displaystyle
\coprod_{1 \leq r < s \leq n+1} \Y(P_r/P_s)$ is naturally isomorphic to
$\End_\calo(\widetilde{\bp})$ as a finite-dimensional $\Z_+$-graded
$\F$-algebra.
\end{prop}

Finally, we define the category of diagrams.

\begin{defn}
Denote by $\ycat$ the category defined by the following structure:
\begin{enumerate}
\item The objects are finite disjoint unions of STYTs.
\item The morphisms are as in Definition \ref{Dstyt}; thus, $\ycat$ is
$\F$-linear.
\item There is a duality functor $(\cdot)^T : \ycat \to \ycat$ that
squares to the identity.
\item Extensions between objects are defined as in Definition
\ref{Dstyt}.
\end{enumerate}
\end{defn}

Notice that every object of $\ycat$ is a (possibly disconnected)
sub-diagram of $\Y_k$ for some $k \geq 1$. The analysis after Definition
\ref{Dstyt} now shows that triangular GWAs categorify Young diagrams.

\begin{prop}
Let $\mathscr{P}$ denote the full subcategory of the block
$\calo[\lambda]$ whose objects are $\{ P_r / P_s : 1 \leq r < s \leq n+1
\}$. Then the assignment $\Y(\cdot)$ is a covariant additive functor from
$\mathscr{P}$ to $\ycat$ that respects morphisms and duality.
\end{prop}

As Theorem \ref{Tyst} suggests, there are other objects on which the
functor $\Y(\cdot)$ respects additional structure. For instance,
morphisms and extensions between Verma modules and dual Verma modules, or
between projectives and arbitrary subquotients of tilting modules, are
also respected by $\Y(\cdot)$. Thus, the discussion in this section
naturally leads to the following overarching question, various aspects of
which will be considered in future study.\medskip

\noindent \textbf{Question.}
Construct a larger category $\widetilde{\ycat} \supset \ycat$ of possibly
non-planar diagrams ``glued'' along edges, and define a functor $\Y(\cdot)
: \calo[\lambda] \to \widetilde{\ycat}$, such that the following
properties are satisfied:
\begin{enumerate}
\item $\Y(\cdot)$ restricts to the functor $\Y(\cdot)$ studied above,
when applied to subquotients of $T_n = P_1$.

\item $\widetilde{\ycat}$ is an $\F$-linear category, equipped with
morphisms, extensions, and a duality functor, which extend to
$\widetilde{\ycat}$ their counterparts in $\ycat$.

\item The functor $\Y(\cdot)$ is exact, and also respects extensions and
duality between objects of $\calo[\lambda]$. Thus, $\Y(F(N)) = \Y(N)^T$,
and Equation \eqref{Ecdsyt} holds for all $N,N'$ in $\calo[\lambda]$.
\end{enumerate}

\noindent This question has obvious connections to the representation
type of the module category $\calo[\lambda]$ (see e.g.~\cite{FNP} for an
analysis in a parallel setting).
Note that the exactness of $\Y(\cdot)$ is also natural to
expect. For instance, the short exact sequence in Equation
\eqref{Etiltsub} has a combinatorial counterpart, as does Equation
\eqref{Eses}:
\[
\emptyset \to \Y(P_r / P_s) \mapdef{\Y(f_{r,s}^{++})} \Y(P_{r+1} / P_{s+1})
\to \Y(M_s / M_r)^T \to \emptyset.
\]

\subsection*{Acknowledgments}

The first author thanks Tom Braden and Daniel Bump for useful
discussions. We would also like to thank the referee for a careful
reading of the paper and for useful comments that improved the
exposition.



\end{document}